\documentclass[12pt,reqno]{amsart}
\usepackage{amsfonts}

\usepackage{amssymb,amsmath,graphicx,amsfonts,euscript}
\usepackage{color}

\newtheorem{thm}{Theorem}[section]

\newtheorem{Pros}{Proposition}[section]
\newtheorem{lemma}{Lemma}[section]

\theoremstyle{definition}
\newtheorem{define}{Definition}[section]

\theoremstyle{remark}
\newtheorem{rem}{Remark}[section]

\allowdisplaybreaks \voffset=-0.2in \numberwithin{equation}{section}

\begin{document}
\bigskip

\centerline{\Large\bf  Global well-posedness for the 2D Euler-Boussinesq-B$\rm\acute{e}$nard }
\smallskip

\centerline{\Large\bf   equations with critical dissipation }

\bigskip

\centerline{Zhuan Ye}

\medskip
\centerline{Department of Mathematics and Statistics, Jiangsu Normal University,
}
\medskip

\centerline{101 Shanghai Road, Xuzhou 221116, Jiangsu, PR China}

\medskip

\centerline{E-mail: \texttt{yezhuan815@126.com
}}

\medskip
\bigskip

{\bf Abstract:}~~%
This present paper is dedicated to the study of the Cauchy problem of the two-dimensional Euler-Boussinesq-B$\rm\acute{e}$nard equations which couple the incompressible Euler equations for the velocity and a transport equation with critical dissipation for the temperature. We show that there is a global unique solution to this model with Yudovich's type data. This settles the global regularity problem which was remarked by Wu and Xue (J. Differential Equations 253:100--125, 2012).

{\vskip 1mm
 {\bf AMS Subject Classification 2010:}\quad 35Q35; 35B65; 35R11; 76B03.

 {\bf Keywords:}
B$\rm\acute{e}$nard equations; Boussinesq equations; Global regularity; Uniqueness; Yudovich's type data.}

\vskip .3in
\section{Introduction and main results}
The two-dimensional (2D) incompressible Euler-Boussinesq-B$\rm\acute{e}$nard equations with critical dissipation for the temperature can be represented in the form
\begin{equation}\label{BEBou}
\left\{\aligned
&\partial_{t}u+(u \cdot \nabla) u+\nabla p=\theta e_{2},\,\,\,\,\,\,\,x\in \mathbb{R}^{2},\,\,t>0, \\
&\partial_{t}\theta+(u \cdot \nabla) \theta+\Lambda \theta=\gamma u_{2}, \\
&\nabla\cdot u=0,\\
&u(x, 0)=u_{0}(x),  \quad \theta(x,0)=\theta_{0}(x),
\endaligned\right.
\end{equation}
where $u(x,\,t)=(u_{1}(x,\,t),\,u_{2}(x,\,t))$ is a vector field denoting
the velocity, $\theta=\theta(x,\,t)$ is a scalar function denoting
the temperature and $p$ is the scalar pressure. The forcing term $\theta e_{2}$ in the velocity equation models the acting of the buoyancy force on the fluid motion, $e_{2}=(0,\,1)$ is the canonical vector and $\gamma\geq0$ is a non-dimensional constant. The first equation $\eqref{BEBou}_{1}$ is the 2D Euler equation with a buoyancy forcing and the second equation $\eqref{BEBou}_{2}$ is a transport-diffusion equation forced by the second component of the velocity. The fractional Laplacian
$\Lambda\triangleq(-\Delta)^{\frac{1}{2}}$ is defined through the Fourier transform, namely
$$\widehat{\Lambda
f}(\xi)=|\xi|\hat{f}(\xi).$$
This fractional Laplacian allows us to study a family of equations simultaneously and may be physically relevant. Actually, there are geophysical circumstances in which \eqref{BEBou} with fractional Laplacian may arise. Flows in the middle
atmosphere traveling upward undergo changes due to the changes of atmospheric
properties, although the incompressibility and Boussinesq approximations are applicable. The effect of kinematic and thermal diffusion is attenuated by the thinning of atmosphere. This anomalous attenuation can be modeled by using the space fractional Laplacian (see e.g., \cite{Constantin,Gill} and references therein). In view of the underlying scaling invariance associated with \eqref{BEBou} (by regarding $\gamma=0$) and the available $L^{\infty}$-maximum principle, it is customary to refer to the dissipation $\Lambda$ as critical. The B$\rm\acute{e}$nard equations describe the Rayleigh-B$\rm\acute{e}$nard convective motion in a heated 2D inviscid incompressible fluid under thermal effects (see e.g. \cite{AP,Chan81,FJT,FMT,MW,Ra}). Moreover, the B$\rm\acute{e}$nard system is also one of the most popular test-problems for comparing numerical algorithms, see  \cite{ATGKZMH} for example. The mathematical analysis of the B$\rm\acute{e}$nard convection system has been studied in Foias et al. \cite{FMT} (see also \cite{RTemam97}), where the existence and uniqueness of
weak solutions in dimension two and three were proved, along with the existence of a finite-dimensional global attractor in space dimension two. Since then, many important results concerning the B$\rm\acute{e}$nard convection system have been obtained. We refer the reader to \cite{ATGKZMH,CJEWNO,DP2,FJT,FMT,FLTJNS,MW,WuXue2012,Yena17} for recent (analytical and numerical) studies concerning the B$\rm\acute{e}$nard convection system.

\vskip .1in
In the case when $\gamma=0$, the system \eqref{BEBou} reduces to the corresponding 2D critical Euler-Boussinesq equations
\begin{equation}\label{Bouswqa21}
\left\{\aligned
&\partial_{t}u+(u \cdot \nabla) u+\nabla p=\theta e_{2},\,\,\,\,\,\,\,x\in \mathbb{R}^{2},\,\,t>0, \\
&\partial_{t}\theta+(u \cdot \nabla) \theta+\Lambda \theta=0, \\
&\nabla\cdot u=0,\\
&u(x, 0)=u_{0}(x),  \quad \theta(x,0)=\theta_{0}(x).
\endaligned\right.
\end{equation}
The Boussinesq equations are of relevance to study a number of models coming from atmospheric or oceanographic turbulence where rotation and stratification play an important role. Moreover, the Boussinesq equations also play an important role in the study of Raleigh-B$\rm\acute{e}$nard convection and share with some key features of the three-dimensional Euler and Navier--Stokes equations such as the vortex stretching mechanism (see, e.g.,\cite{ConDoering,Gill,MB,PG}). On the one hand, the question of whether the 2D completely inviscid Boussinesq equations (namely \eqref{Bouswqa21} without the dissipation term $\Lambda \theta$) can develop a finite-time singularity from general initial data is a challenging open problem. On the other hand, if moderately strong dissipation is added in the equation of $\theta$, then 2D Euler-Boussinesq type equations indeed admit a unique global solution. In fact, if the dissipation term $\Lambda \theta$ in $\eqref{Bouswqa21}_{2}$ is replaced by $-\Delta\theta$, Chae \cite{Chae2} proved the global well-posedness for this case with regular initial data. For the subsequent works with the rough initial data, we refer to \cite{HKbmp1,DP2}. In the case when the dissipation term $\Lambda \theta$ in $\eqref{Bouswqa21}_{2}$ is replaced by $\Lambda^{\beta}\theta$ with $\beta\in (1,2)$, Hmidi and Zerguine \cite{HZ10} showed the global well-posedness of the corresponding system with the rough data. The global well-posedness of \eqref{Bouswqa21} was proven by Hmidi, Keraani and Rousset \cite{HK4} via deeply developing the hidden structures of the coupling system. Recently, the author \cite{Yena} provided an alternative approach to global regularity for \eqref{Bouswqa21} without the use of the hidden structure developed by Hmidi, Keraani and Rousset \cite{HK4}. Actually, the main argument of \cite{Yena} is mainly based on the sharp interpolation inequalities and the differentiability of the drift--diffusion equation. We also refer to \cite{DPqq,Hmdapde,KCRTW,XXue1} for some other interesting results of the Euler-Boussinesq type equations.

\vskip .1in
In terms of (\ref{BEBou}) with $\gamma>0$, it not only has the difficulties that \eqref{Bouswqa21} possesses, but also the external force $u_{2}$ will bring other greater difficulties. Actually, in the case when $\gamma>0$, the global regularity issue of (\ref{BEBou}) becomes much more involved and harder than \eqref{Bouswqa21}. The key reason is that the external force $u_{2}$ in the temperature equation prevents us to get the {\it a priori} $L^{\infty}$-information of $\theta$. In fact, the energy method and classical approaches are no longer effective in obtaining $L^{\infty}$-bound of $\theta$. Indeed, it seems impossible to directly apply the methods of \cite{HK4} and \cite{Yena} to establish the global regularity of (\ref{BEBou}) with $\gamma>0$. Let us now briefly explain that what difficulties the external force $u_{2}$ in $\eqref{BEBou}_{2}$ will bring. As $u$ satisfies a pure transport equation forced by $\theta e_{2}$, the energy method only provides the boundedness information $u\in L_{loc}^{\infty}(\mathbb{R}_{+}; L^{2}(\mathbb{R}^{2}))$. Keeping in mind this bound and applying the delicate analysis on the temperature equation $(\ref{BEBou})_{2}$, we are only able to obtain the following polynomial control of $\theta$ (see \eqref{yydkl028} below for details)
\begin{eqnarray}\label{hkjd2}
 \sup_{r\geq2}\frac{\|\theta(\tau)\|_{L^{r}}}{r} <\infty.
\end{eqnarray}
Unfortunately, with just the above estimate \eqref{hkjd2} and the known arguments, it seems very difficult to further derive any regularity estimates of the solution (see \cite{Yena17,Yeacap18} for details). To explain clearly, we provide two examples to illustrate it. On the one hand, if one follows the approach of \cite{HK4}, then we may derive
 \begin{align}\label{hfdcvv11}&
 \frac{d}{dt}(\|\omega(t)\|_{L^{2}}^{2}+\|\omega(t)\|_{L^{4}}^{4}+\|\nabla\theta(t)\|_{L^{2}}^{2})
+\|\Lambda^{\frac{3}{2}}\theta\|_{L^{2}}^{2}
\nonumber\\&\leq C(1+\|\Lambda^{\frac{1}{2}}\theta\|_{L^{2}}^{2}+\|\theta\|_{B_{\infty,2}^{0}})
(\|\omega\|_{L^{2}}^{2}+
\|\omega\|_{L^{4}}^{4}+
\|\nabla \theta\|_{L^{2}}^{2}),
\end{align}
where $\omega$ denotes the vorticity associated with the corresponding velocity $u$, namely, $\omega\triangleq \nabla^{\perp}\cdot u=\partial_{1}u_{2}-\partial_{2}u_{1}$.
We point out that \eqref{hfdcvv11} can be deduced as follows.
It follows from the proof of \cite[Proposition 5.2]{HK4} that
\begin{align}\label{cvkkjh1}
 \frac{d}{dt}(\|\omega(t)\|_{L^{2}}^{2}+\|\omega(t)\|_{L^{4}}^{4})
 \leq C(1+\|\theta\|_{B_{\infty,2}^{0}})
(\|\omega\|_{L^{2}}^{2}+
\|\omega\|_{L^{4}}^{4}).
\end{align}
Multiplying $(\ref{BEBou})_{2}$ by $-\Delta\theta$ and performing a space integration by parts, we get
\begin{align}
\frac{1}{2}\frac{d}{dt}\|\nabla\theta(t)\|_{L^{2}}^{2}
+\|\Lambda^{\frac{3}{2}}\theta\|_{L^{2}}^{2}&=\int_{\mathbb{R}^{2}}
 \big(u \cdot
\nabla\theta\big)\Delta\theta\,dx-\int_{\mathbb{R}^{2}}
 u_{2}\Delta\theta\,dx\nonumber\\
&=-\int_{\mathbb{R}^{2}}
\partial_{j}u_{i} \partial_{i}\theta \partial_{j}\theta\,dx+\int_{\mathbb{R}^{2}}
\partial_{i} u_{2}\partial_{i}\theta\,dx\nonumber\\
& \leq C\|\nabla u\|_{L^{4}}\|\nabla\theta\|_{L^{\frac{8}{3}}}^{2}+C\|\nabla u\|_{L^{2}}\|\nabla \theta\|_{L^{2}}
\nonumber\\
& \leq C\|\omega\|_{L^{4}}\|\Lambda^{\frac{5}{4}}\theta\|_{L^{2}}^{2}
+C\|\omega\|_{L^{2}}\|\nabla \theta\|_{L^{2}}
\nonumber\\
& \leq C\|\omega\|_{L^{4}}\|\Lambda^{\frac{1}{2}}\theta\|_{L^{2}}^{\frac{1}{2}}
\|\Lambda^{\frac{3}{2}}\theta\|_{L^{2}}^{\frac{3}{2}}+C(\|\omega\|_{L^{2}}^{2}+
\|\nabla \theta\|_{L^{2}}^{2})
\nonumber\\
& \leq
\frac{1}{2}\|\Lambda^{\frac{3}{2}}\theta\|_{L^{2}}^{2}
+C\|\Lambda^{\frac{1}{2}}\theta\|_{L^{2}}^{2}\|\omega\|_{L^{4}}^{4}+C(\|\omega\|_{L^{2}}^{2}+
\|\nabla \theta\|_{L^{2}}^{2})
\nonumber\\
& \leq
\frac{1}{2}\|\Lambda^{\frac{3}{2}}\theta\|_{L^{2}}^{2}
+C(1+\|\Lambda^{\frac{1}{2}}\theta\|_{L^{2}}^{2})(
\|\omega\|_{L^{4}}^{4}+\|\omega\|_{L^{2}}^{2}+
\|\nabla \theta\|_{L^{2}}^{2}),\nonumber
\end{align}
which yields
\begin{align}\label{dfpyt78}
 \frac{d}{dt}\|\nabla\theta(t)\|_{L^{2}}^{2}
+\|\Lambda^{\frac{3}{2}}\theta\|_{L^{2}}^{2}
\leq C(1+\|\Lambda^{\frac{1}{2}}\theta\|_{L^{2}}^{2})(
\|\omega\|_{L^{4}}^{4}+\|\omega\|_{L^{2}}^{2}+
\|\nabla \theta\|_{L^{2}}^{2}).
\end{align}
Therefore, \eqref{hfdcvv11} follows by combining \eqref{cvkkjh1} and \eqref{dfpyt78}. On the other hand, arguing as \cite[(2.10)]{Yena}, we are able to show
 \begin{align}\label{hfdcvv12}&
 \frac{d}{dt}(\|\omega(t)\|_{L^{2}}^{2}+\|\omega(t)\|_{L^{3}}^{3}
 +\|\nabla\theta(t)\|_{L^{2}}^{2})
+\|\Lambda^{\frac{3}{2}}\theta\|_{L^{2}}^{2}
\nonumber\\&\leq C(1+\|\theta\|_{L^{\infty}}^{2})
(\|\omega\|_{L^{2}}^{2}+
\|\omega\|_{L^{3}}^{3}+
\|\nabla \theta\|_{L^{2}}^{2}).
\end{align}
Applying the corresponding logarithmic Sobolev inequalities to bound $\|\theta\|_{B_{\infty,2}^{0}}$ of \eqref{hfdcvv11} and $\|\theta\|_{L^{\infty}}^{2}$ of \eqref{hfdcvv12}, we are still unable to establish the desired differential type inequalities due to the number $2$ in $\|\theta\|_{B_{\infty,2}^{0}}$ and $\|\theta\|_{L^{\infty}}^{2}$.
As a matter of fact, if one expects to close \eqref{hfdcvv11} or \eqref{hfdcvv12}, precisely to derive the desired differential type inequalities, then the following bound ("almost the same level") is strongly required (see the proof of \cite[Theorem 1.3]{Yeacap18})
\begin{eqnarray}\label{hkjd1}
\sup_{r\geq2}\frac{\|\theta(\tau)\|_{L^{r}}}{\sqrt{r}}<\infty.
\end{eqnarray}
In fact, it is very difficult to fulfil the gap between \eqref{hkjd2} and \eqref{hkjd1}. However, if $\gamma=0$, we do not need to face this problem as the $L^{\infty}$-information of $\theta$ is automatically true due to the maximum principle. In the case when the dissipation $\Lambda \theta$ is replaced by $-\Delta\theta$, Danchin and Paicu \cite{DP2} proved that the B$\rm\acute{e}$nard equations (\ref{BEBou}) with $\gamma>0$ (of course valid for $\gamma=0$) have a global unique solution for Yudovich's type data, that is, the initial data has finite energy and bounded vorticity. Moreover, they also showed the global result in the case of infinite energy velocity field which can admit the vortex-patches-like structures. Later, in the case when $\Lambda \theta$ is replaced by $\Lambda^{\beta}\theta$ with $\beta\in (1,2)$ and $\gamma>0$ (also valid for $\gamma=0$), Wu and Xue \cite{WuXue2012} were able to show the global unique solution for the corresponding B$\rm\acute{e}$nard equations with Yudovich's type data. However, completely different from the corresponding critical Boussinesq equations \cite{HK4}, to the best of our knowledge, the global regularity of B$\rm\acute{e}$nard equations (\ref{BEBou}) with $\gamma>0$ is still an unsolved problem, see \cite{WuXue2012,Yena17,Yeacap18} for partial results. Interestingly, it was stated in \cite[Remark 1.3]{WuXue2012} that it is not clear how to show the global regularity of (\ref{BEBou}) with $\gamma>0$. As a matter of fact, this is the main aim of this paper. Based on a new commutator estimate and a refined logarithmic Gronwall inequality as well as the nice combined quantity, we are able to give an affirmative solution to this unsolved critical problem (\ref{BEBou}) with Yudovich's type data. More precisely, we aim at proving the following global well-posedness result.

\begin{thm}\label{Th1}
Let $\gamma>0$, $u_{0}\in L^{2}(\mathbb{R}^{2})$ be a divergence-free vector field, the corresponding initial
vorticity $\omega_{0}\in L^{2}(\mathbb{R}^{2})\cap L^{\infty}(\mathbb{R}^{2})$ and $\mathcal {R}_{1}\theta_{0}\in L^{\infty}(\mathbb{R}^{2})$, $\theta_{0}\in  \mathbb{L}(\mathbb{R}^{2})$.
Then (\ref{BEBou}) generates a unique global solution $( {u},\, \theta)$ such that
$$u\in \mathcal{C}(\mathbb{R}_{+}; H^{1}(\mathbb{R}^{2})),\quad \omega \in L_{loc}^{\infty}(\mathbb{R}_{+}; L^{\infty}(\mathbb{R}^{2})),\quad \theta \in L_{loc}^{\infty}(\mathbb{R}_{+};\mathbb{L}(\mathbb{R}^{2})),$$
$$ \mathcal {R}_{1}\theta \in L_{loc}^{\infty}(\mathbb{R}_{+}; L^{\infty}(\mathbb{R}^{2})),\quad \theta\in \mathcal{C}(\mathbb{R}_{+}; B_{r,1}^{\sigma}(\mathbb{R}^{2}))\cap L_{loc}^{1}(\mathbb{R}_{+}; B_{r,1}^{\sigma+1}(\mathbb{R}^{2})),$$
where $\sigma$ and $r$ satisfy
$$\frac{2}{r}-1<\sigma<0,\quad 2<r<\infty.$$
Here $\mathcal {R}_{1}$ is the classical 2D Riesz operator, namely, $\mathcal {R}_{1}\triangleq \Lambda^{-1}\partial_{x_{1}}$ and the space $\mathbb{L}(\mathbb{R}^{2})$ is defined by
$$\mathbb{L}(\mathbb{R}^{2})=\left\{f\in \mathcal{S'}(\mathbb{R}^{2}):\ \ \sup_{m\geq2}\frac{\|f\|_{L^{m}(\mathbb{R}^{2})}}{m}<\infty\right\}.$$
\end{thm}

\vskip .1in
\begin{rem}
Since $\omega \in L_{loc}^{\infty}(\mathbb{R}_{+}; L^{\infty}(\mathbb{R}^{2}))$, we are able to propagate all the higher regularities, for example $u_{0}\in H^{s}(\mathbb{R}^{2}),\,\theta_{0}\in H^{\widetilde{s}}(\mathbb{R}^{2})$ with $|s-\widetilde{s}|\leq\frac{1}{2}$ and $s>2$ (see Appendix \ref{appSec2} for details). Moreover, we assume $\gamma=1$ for the sake of simplicity. It is worthwhile to stress one point that Theorem \ref{Th1} is valid for any $\gamma \in \mathbb{R}$. Finally, the restriction $\sigma<0$ is only used to ensure $\mathbb{L}(\mathbb{R}^{2})\hookrightarrow B_{r,1}^{\sigma}(\mathbb{R}^{2})$.
\end{rem}

\begin{rem}
For the Euler-Boussinesq equations \eqref{Bouswqa21}, we \cite{Yena} provided an alternative approach to global regularity without the use of the combined quantity $\Gamma$ (see \eqref{fgt504} for details). However, we have no idea how to prove the global regularity for the Euler-Boussinesq-B$\rm\acute{e}$nard equations (\ref{BEBou}) without the use of such combined quantity.
\end{rem}

\begin{rem}
Combining the known results on the 2D hydrodynamic equations and our argument used in the present paper, the same global regularity result is also valid for other 2D hydrodynamic system involving the Euler-B$\rm\acute{e}$nard equations.
For example, the following 2D magnetic Euler-B$\rm\acute{e}$nard equations come from the convective motions in a heated and incompressible fluid, which can be used to model the heat convection phenomenon under the influence of the magnetic field  (see, e.g., \cite{Gpadd,Mulone} for more details)
\begin{equation}\label{dsfpmg6}
\left\{\aligned
&\partial_{t}u+(u\cdot\nabla)u+\nabla p=(b\cdot\nabla) b+\theta e_{2},\qquad x \in \mathbb{R}^{2},\, t>0,\\
&\partial_{t}b+(u\cdot\nabla)b+\Lambda^{2\beta}b=(b\cdot\nabla)u,
\\
&\partial_{t}\theta+(u \cdot \nabla) \theta+\Lambda \theta= u_{2}, \\
&\nabla\cdot u=0,\ \ \ \nabla\cdot b=0,\\
&u(x,0)=u_{0}(x),\,\,b(x,0)=b_{0}(x),\,\,\theta(x,0)=\theta_{0}(x).
\endaligned \right.
\end{equation}
Making use of the arguments adopted in proving Theorem \ref{Th1} and performing some suitable modifications, we are able to show that in the case when $\beta>1$, \eqref{dsfpmg6} admits a unique global smooth solution. However, it should be mentioned here that we are unable to show the global regularity for \eqref{dsfpmg6} when $\Lambda^{2\beta}b$ with $\beta>1$ is replaced by the weaker dissipation $-\Delta [\ln(e-\Delta)]^{\rho}b$ with $\rho\geq0$.  Therefore, it would be interesting to study this case and this is left for the future.
Furthermore, the global regularity result also holds true for the corresponding 2D micropolar Rayleigh-Euler-B$\rm\acute{e}$nard convection problem \cite{Eringen,Lukasz1,Lukasz2}
\begin{equation*}
\left\{\aligned
&\partial_{t}u+(u\cdot\nabla)u+\nabla p=\nabla\times w+\theta e_{2},\qquad x \in \mathbb{R}^{2},\, t>0,\\
&\partial_{t}w+(u\cdot\nabla)w-\Delta w=\nabla^{\perp}\cdot u,
\\
&\partial_{t}\theta+(u \cdot \nabla) \theta+\Lambda \theta= u_{2}, \\
&\nabla\cdot u=0, \\
&u(x,0)=u_{0}(x),\,\,w(x,0)=w_{0}(x),\,\,\theta(x,0)=\theta_{0}(x),
\endaligned \right.
\end{equation*}
where $u$ is the fluid velocity, $w$ is the field of microrotation representing the angular velocity of the
rotation of the particles of the fluid, $p$ is the scalar pressure of the flow, $\theta$ is the scalar temperature and $\nabla\times w=(-\partial_{2}w,\,\partial_{1}w)$.
\end{rem}

\vskip .2in
The rest of this paper unfolds as follows. Section \ref{sectt2} is devoted to providing the definition of Besov spaces and collecting some inequalities which will be needed in later analysis. In Section \ref{sectt011}, we mainly prove Theorem \ref{Th1}, which consists of three steps, namely, the \emph{a priori} estimates, existence and the uniqueness. Finally, in Appendix \ref{appSec2}, we consider the case of the system (\ref{BEBou}) with regular initial data.

\vskip .3in
\section{Preliminaries}\setcounter{equation}{0}\label{sectt2}
This section provides the definition of Besov spaces and
several useful facts. Let us recall briefly the definition of the
Littlewood-Paley decomposition (see \cite{BCD} for more details). More precisely, we take some smooth radial non increasing function $\chi$ with values in $[0, 1]$ such that $\chi\in C_{0}^{\infty}(\mathbb{R}^{2})$ supported in the ball $\mathcal{B}\triangleq\{\xi\in \mathbb{R}^{2}, |\xi|\leq \frac{4}{3}\}$
and with value $1$ on $\{\xi\in \mathbb{R}^{2}, |\xi|\leq \frac{3}{4}\}$, then we define
$\varphi(\xi)\triangleq\chi\big(\frac{\xi}{2}\big)-\chi(\xi)$. Of course, it is not hard to check that
${\varphi\in C_{0}^{\infty}(\mathbb{R}^{2})}$ is supported in the annulus
$\mathcal{C}\triangleq\{\xi\in \mathbb{R}^{2}, \frac{3}{4}\leq |\xi|\leq
\frac{8}{3}\}$ and satisfies
$$\chi(\xi)+\sum_{j\geq0}\varphi(2^{-j}\xi)=1, \quad  \forall \xi\in \mathbb{R}^{2};\qquad \sum_{j\in \mathbb{Z}}\varphi(2^{-j}\xi)=1, \quad  \forall \xi\neq0.$$
Let $h=\check{\varphi}$ and $\widetilde{h}=\check{\chi}$, then
the inhomogeneous dyadic blocks $\Delta_{j}$ of our decomposition by setting
$$\Delta_{j}u=0,\ \ j\leq -2; \ \  \ \ \ \Delta_{-1}u=\chi(D)u=\int_{\mathbb{R}^{2}}{\widetilde{h}(y)u(x-y)\,dy};
$$
$$ \Delta_{j}u=\varphi(2^{-j}D)u=2^{jn}\int_{\mathbb{R}^{2}}{h(2^{j}y)u(x-y)\,dy},\ \ \forall j\geq0.
$$
Moreover, the low-frequency cut-off reads
$$ {S}_{j}u=\chi(2^{-j}D)u=\sum_{k\leq j-1} {\Delta}_{k}u=2^{jn}\int_{\mathbb{R}^{2}}{\widetilde{h}(2^{j}y)u(x-y)\,dy},\ \ \forall j\in \mathbb{N}.$$
Moreover, we have the following so-called Bony decomposition
$$uv=\sum_{j\geq -1} {S}_{j-1}u  {\Delta}_{j}v+ \sum_{j\geq -1} {S}_{j-1}v {\Delta}_{j}u+ \sum_{j\geq -1}{\Delta}_{j}u\widetilde{\Delta}_{j}v,$$
with
$$\widetilde{\Delta}_{j}v={\Delta}_{j-1}v+
 {\Delta}_{j}v+ {\Delta}_{j+1}v.$$

\vskip .1in
It is now time to introduce the inhomogeneous Besov spaces, which are defined by the Littlewood-Paley decomposition.
\begin{define}
Let $s\in \mathbb{R}, (p,r)\in[1,+\infty]^{2}$. The inhomogeneous
Besov space $B_{p,r}^{s}$ is defined as a space of $f\in
S'(\mathbb{R}^{2})$ such that
$$ B_{p,r}^{s}=\{f\in S'(\mathbb{R}^{2});  \|f\|_{B_{p,r}^{s}}<\infty\},$$
where
\begin{equation} \nonumber
 \|f\|_{B_{p,r}^{s}}\triangleq\left\{\aligned
&\Big(\sum_{j\geq-1}2^{jrs}\|\Delta_{j}f\|_{L^{p}}^{r}\Big)^{\frac{1}{r}}, \quad \ r<\infty,\\
&\sup_{j\geq-1}
2^{js}\|\Delta_{j}f\|_{L^{p}}, \quad \ r=\infty.\\
\endaligned\right.
\end{equation}
\end{define}

\vskip .1in
The following one is the Bernstein type inequality (see \cite[Lemma 2.1]{BCD}).
\begin{lemma} \label{vfgty8xc}
Assume $1\leq a\leq b\leq\infty$ and $\mathcal{C}$ be an annulus and $\mathcal{B}$ a ball of $\mathbb{R}^{2}$, then
$${\rm Supp}\widehat{f}\subset \lambda \mathcal{B}\ \Rightarrow\
\|\nabla^{k}f\|_{L^b}\leq C_1\, \lambda^{k  +
2(\frac{1}{a}-\frac{1}{b})} \|f\|_{L^a},
$$
$$  {\rm Supp}\widehat{f}\subset \lambda \mathcal{C}\ \Rightarrow\
C_2\, \lambda^{k} \| f\|_{L^b}\leq \|\nabla^{k} f\|_{L^b}\leq
C_3\, \lambda^{k + 2(\frac{1}{a}-\frac{1}{b})} \|f\|_{L^a},
$$
where $C_1$, $C_2$ and $C_3$ are constants depending on $k,\,a$ and $b$
only.
\end{lemma}

\vskip .1in

To prove our theorem, we need the following commutator estimate involving the operator $\mathcal {R}_{1}$.
\begin{lemma}\label{commutatorq1}
Let $1<p<\infty$, then it holds
\begin{align}\label{sdfvh80}
 \|[\mathcal {R}_{1},u\cdot\nabla]\theta\|_{L^{p}}\leq C(p)\|\nabla u\|_{L^{p}}\|\theta\|_{L^{\infty}},
\end{align}
where $[A, B]\triangleq AB-BA$ denotes the commutator operator.
\end{lemma}
\begin{proof}
We rewrite the commutator as
\begin{align}\label{sdfvh81}
[\mathcal {R}_{1},u\cdot\nabla]\theta
=\mathcal {R}_{1}\partial_{j}(u_{j}\theta)
-u_{j}\mathcal {R}_{1}\partial_{j}\theta-\mathcal {R}_{1}(\nabla\cdot u\, \theta),
\end{align}
where here and in what follows we use the Einstein summation convention. We point out that the third term at the right hand side of \eqref{sdfvh81} vanishes when $u$ is a divergence free vector field.
Obviously, one has
\begin{align}\label{sdfvh82}
 \|\mathcal {R}_{1}(\nabla\cdot u\, \theta)\|_{L^{p}}\leq C(p)\| \nabla\cdot u \,\theta\|_{L^{p}}\leq C(p)\|\nabla u\|_{L^{p}}\|\theta\|_{L^{\infty}},
\end{align}
where here and in what follows we frequently use the fact
$$\|\mathcal {R}_{1}f\|_{L^{p}}\leq C\frac{p^{2}}{p-1}\|f\|_{L^{p}}.$$
Denoting $A_{j}=\mathcal {R}_{1}\partial_{j}$ for $j=1,\,2$, we thus have
\begin{align}
\mathcal {R}_{1}\partial_{j}(u_{j}\theta)
-u_{j}\mathcal {R}_{1}\partial_{j}\theta=&A_{j}(u_{j}\theta)
-u_{j}A_{j}\theta\nonumber\\
=&\underbrace{A_{j}(u_{j}\theta)
-u_{j}A_{j}\theta-\theta
A_{j}u_{j}}_{N_{1}} +\underbrace{\theta
A_{j}u_{j}}_{N_{2}}.\nonumber
\end{align}
Of course, it implies
\begin{align} \label{sdfvh83}
 \|N_{2} \|_{L^{p}}&\leq C\|A_{j}u_{j}\|_{L^{p}}\|\theta\|_{L^{\infty}}
\nonumber\\&\leq
C(p)\|\nabla u\|_{L^{p}}\|\theta\|_{L^{\infty}}.
\end{align}
Choosing $s=s_{1}=1,\,s_{2}=0$ in (1.10) of \cite[Corollary 1.4]{Lirmi19} or \cite[Corollary 5.4]{Lirmi19}, we are able to derive
\begin{align}\label{sdfvh84}
 \|N_{1} \|_{L^{p}} &\leq
C(p)\|\Lambda u\|_{L^{p}}\|\theta\|_{{\rm{BMO}}}\nonumber\\&\leq
C(p)\|\nabla u\|_{L^{p}}\|\theta\|_{L^{\infty}},
\end{align}
where $\rm{BMO}$ denotes the bounded mean oscillation space.
Combining \eqref{sdfvh83} and \eqref{sdfvh84}, we get
\begin{align}
 \|\mathcal {R}_{1}\partial_{j}(u_{j}\theta)
-u_{j}\mathcal {R}_{1}\partial_{j}\theta\|_{L^{p}}\leq
C(p)\|\nabla u\|_{L^{p}}\|\theta\|_{L^{\infty}},\nonumber
\end{align}
which together with \eqref{sdfvh81} and \eqref{sdfvh82} yields
\begin{align}
 \|[\mathcal {R}_{1},u\cdot\nabla]\theta\|_{L^{p}}\leq C(p)\|\nabla u\|_{L^{p}}\|\theta\|_{L^{\infty}}.\nonumber
\end{align}
Consequently, we complete the proof of Lemma \ref{commutatorq1}.
\end{proof}

\vskip .1in
Next, we recall the following lemma concerning the so called Kato-Ponce commutator estimate, which can be found in many references, e.g., \cite{TKG,Kenig,Lirmi19}.
\begin{lemma}
Let $p, p_1, p_3\in (1, \infty)$ and $p_2, p_4\in [1,\infty]$ satisfy
$$
\frac1p =\frac1{p_1} + \frac1{p_2} = \frac1{p_3} + \frac1{p_4}.
$$
Then for $s>0$, there exists a positive constant $C$ such that
\begin{align}\label{cmoi78ye1}
\|[\Lambda^s, f]g\|_{L^p} \le C \left(\|\Lambda^{s-1} g\|_{L^{p_1}}\,\|\nabla f\|_{L^{p_2}}+\|\Lambda^s f\|_{L^{p_3}}\, \|g\|_{L^{p_4}} \right).
\end{align}
In particular, it yields
\begin{align}\label{cmoi78ye2r}
\|\Lambda^s(fg)\|_{L^p} \le C \left(\|\Lambda^{s} g\|_{L^{p_1}}\,\| f\|_{L^{p_2}}+\|\Lambda^s f\|_{L^{p_3}}\, \|g\|_{L^{p_4}} \right).
\end{align}
\end{lemma}

\vskip .1in
Finally, the following refined logarithmic Gronwall inequality plays an important role in proving our main result (see \cite{CLTPD,Yeacap18}).
\begin{lemma}\label{Lem01khj}
Assume that $l(t),\,m(t),\,n(t)$ and $f(t)$ are all nonnegative and integrable functions on $(0, T)$.
Let $A(t)$ and $B(t)$ be two
absolutely continuous and nonnegative functions on $(0, T)$ for any given $T>0$, satisfying for any $t\in (0, T)$
\begin{equation}
A'(t)+B(t)\leq \Big[l(t)+m(t)\ln(A+e)+n(t) \ln(A+B+e)\Big](A+e)+f(t).\nonumber
\end{equation}
If there are three constants $K\in[0,\,\infty)$, $\alpha\in[0,\,\infty)$ and $\beta\in [0,\,1)$ such that for any $t\in (0, T)$
\begin{equation}\label{ympery8}
n(t)\leq K\big(1+A(t)\big)^{\alpha}\big(1+A(t)+B(t)\big)^{\beta},
\end{equation}
then it holds true
\begin{equation}
A(t)+\int_{0}^{t}{B(s)\,ds}\leq C<\infty.\nonumber
\end{equation}
\end{lemma}

\vskip .3in
\section{The proof of Theorem \ref{Th1}}\setcounter{equation}{0}\label{sectt011}
In this section, we give the proof of Theorem \ref{Th1}. Before proving our main result, we point out that the letter $C$ denotes a harmless positive constant, the meaning of which is clear from the context. To emphasize the dependence of a constant on some certain quantities $\alpha,\beta,\cdot\cdot\cdot$, we write
$C(\alpha,\beta,\cdot\cdot\cdot)$.

\vskip .1in
The proof of Theorem \ref{Th1} is divided into three subsections. The first subsection is devoted to showing some key \emph{a priori} estimates. In second subsection, based on the obtained \emph{a priori} estimates, we aim at proving the existence part. The final subsection concerns the uniqueness result.

\subsection{\emph{A priori} estimates}
Let us begin with the following energy estimate.
\begin{Pros} \label{Gsds1}
Assume $(u_{0},\,\theta_{0})$ satisfies the assumptions stated in
Theorem \ref{Th1}, then the corresponding solution $(u,\theta)$ obeys
\begin{eqnarray}\label{yydkl027}
\|u(t)\|_{L^{2}}^{2}+\|\theta(t)\|_{L^{2}}^{2}+\int_{0}^{t}{
\|\Lambda^{\frac{1}{2}}\theta(\tau)\|_{L^{2}}^{2}\,d\tau}\leq
C(t,\,u_{0},\,\theta_{0}),
\end{eqnarray}
\begin{eqnarray}\label{yydkl028}
\sup_{r\geq2}\frac{\|\theta(t)\|_{L^{r}}}{r}\leq C(t, {u}_0,\theta_0).
\end{eqnarray}
\end{Pros}
\begin{proof}
We point out that \eqref{yydkl028} is basically established in \cite[Theorem 1.3]{Yena17}. For the convenience of the reader, the proof will be also sketched here.
Taking the $L^{2}$ inner product of the velocity equation with $u$ and the temperature with $\theta$, we
find by adding them up
\begin{align}\label{ghvber68}
\frac{1}{2}\frac{d}{dt}(\|u(t)\|_{L^{2}}^{2}+\|\theta(t)\|_{L^{2}}^{2})+ \|\Lambda^{\frac{1}{2}}\theta\|_{L^{2}}^{2}=&2\int_{\mathbb{R}^{2}}
{u_{2}\,\theta\,dx} \nonumber\\
 \leq& 2\|u\|_{L^{2}}\|\theta\|_{L^{2}} \nonumber\\
 \leq&
 \|u\|_{L^{2}}^{2}+\|\theta\|_{L^{2}}^{2},
\end{align}
which together with the Gronwall inequality implies
$$\|u(t)\|_{L^{2}}^{2}+\|\theta(t)\|_{L^{2}}^{2}+\int_{0}^{t}{
\|\Lambda^{\frac{1}{2}}\theta(\tau)\|_{L^{2}}^{2}\,d\tau}\leq (\|u_{0}\|_{L^{2}}^{2}+\|\theta_{0}\|_{L^{2}}^{2})e^{2t}.$$
Multiplying $(\ref{BEBou})_{2}$ by
$|\theta|^{r-2}\theta$ and integrating it over whole space, we get
\begin{align}\label{t501}
\frac{1}{r}\frac{d}{dt}\|\theta(t)\|_{L^{r}}^{r}+ \int_{\mathbb{R}^{2}}
\Lambda \theta (|\theta|^{r-2}\theta)\,dx=\int_{\mathbb{R}^{2}}
{u_{2}\,\,(|\theta|^{r-2}\theta)\,dx}.
\end{align}
According to the positivity inequality \cite[Lemma 3.3]{JN} and Sobolev embedding $\dot{H}^{\frac{1}{2}}(\mathbb{R}^{2})\hookrightarrow L^{4}(\mathbb{R}^{2})$, we infer that
$$\int_{\mathbb{R}^{2}}
(\Lambda \theta)|\theta|^{r-2}f\,dx\geq \frac{2}{r}\int_{\mathbb{R}^{2}}
(\Lambda^{\frac{1}{2}}|\theta|^{\frac{r}{2}})^{2}\,dx\geq \frac{\widetilde{c}}{r}\|\theta\|_{L^{2r}}^{r},$$
where $\widetilde{c}>0$ is an absolute constant.
Therefore, one gets from \eqref{t501} that
\begin{align}\label{t502}
\frac{1}{r}\frac{d}{dt}\|\theta(t)\|_{L^{r}}^{r}
+\frac{\widetilde{c}}{r}\|\theta\|_{L^{2r}}^{r}
\leq\int_{\mathbb{R}^{2}}
{u_{2}\,\,(|\theta|^{r-2}\theta)\,dx}.
\end{align}
Owing to the H$\rm {\ddot{o}}$lder inequality, we derive
\begin{align}
\Big|\int_{\mathbb{R}^{2}}
{u_{2}\,\,(|\theta|^{r-2}\theta)\,dx}\Big| \leq&
\|u_{2}\|_{L^{2}}\||\theta|^{r-2}\theta\|_{L^{2}}
\nonumber\\ \leq& \|u\|_{L^{2}}\|\theta\|_{L^{2(r-1)}}^{r-1}\nonumber\\
 \leq& \|u\|_{L^{2}}\|\theta\|_{L^{r}}\|\theta\|_{L^{2r}}^{r-2}.\nonumber
\end{align}
Invoking the $\varepsilon$-Young inequality
$$ab\leq \frac{\varepsilon^{p}}{p}a^{p}+\frac{\varepsilon^{-q}}{q}b^{q}$$
with
$$
a=\|u\|_{L^{2}}\|\theta\|_{L^{r}},\quad b=\|\theta\|_{L^{2r}}^{r-2}, \quad p=\frac{r}{2},\quad q=\frac{r}{r-2}, \quad \varepsilon=\left(\frac{2(r-2)}{\widetilde{c}}\right)^{\frac{r-2}{r}},$$
we have
$$\|u\|_{L^{2}}\|\theta\|_{L^{r}}\|\theta\|_{L^{2r}}^{r-2}\leq \frac{\widetilde{c}}{2r}\|\theta\|_{L^{2r}}^{r}+\frac{2}{r}
\left(\frac{2(r-2)}{\widetilde{c}}\right)^{\frac{r-2}{2}}
\|u\|_{L^{2}}^{\frac{r}{2}}\|\theta\|_{L^{r}}^{\frac{r}{2}},$$
which yields
\begin{align}\label{t503}
\Big|\int_{\mathbb{R}^{2}}
{u_{2}\,\,(|\theta|^{r-2}\theta)\,dx}\Big| \leq&\frac{\widetilde{c}}{2r}\|\theta\|_{L^{2r}}^{r}+\frac{2}{r}
\left(\frac{2(r-2)}{\widetilde{c}}\right)^{\frac{r-2}{2}}
\|u\|_{L^{2}}^{\frac{r}{2}}\|\theta\|_{L^{r}}^{\frac{r}{2}}
\end{align}
Inserting (\ref{t503}) into (\ref{t502}) implies
\begin{align}
\frac{1}{r}\frac{d}{dt}\|\theta(t)\|_{L^{r}}^{r}
&\leq \frac{2}{r}
\left(\frac{2(r-2)}{\widetilde{c}}\right)^{\frac{r-2}{2}}
\|u\|_{L^{2}}^{\frac{r}{2}}\|\theta\|_{L^{r}}^{\frac{r}{2}},\nonumber
\end{align}
which along with \eqref{yydkl027} allows us to get
\begin{align}
\frac{d}{dt}\|\theta(t)\|_{L^{r}}^{\frac{r}{2}}
\leq  \left(\frac{2(r-2)}{\widetilde{c}}\right)^{\frac{r-2}{2}}
\|u\|_{L^{2}}^{\frac{r}{2}} \leq  \left(\frac{2(r-2)}{\widetilde{c}}\right)^{\frac{r-2}{2}}
C(t)^{\frac{r}{2}}.\nonumber
\end{align}
Integrating above inequality in time yields
$$\|\theta(t)\|_{L^{r}}^{\frac{r}{2}}\leq \|\theta_{0}\|_{L^{r}}^{\frac{r}{2}}
+\left(\frac{2(r-2)}{\widetilde{c}}\right)^{\frac{r-2}{2}}
C(t)^{\frac{r}{2}}
.$$
Bearing in mind the fact $(x+y)^{\varepsilon}\leq x^{\varepsilon}+y^{\varepsilon}$ for any $x,\,y\geq0$ and $0\leq \varepsilon\leq1$, we have
\begin{align}\label{xfyped26}
\|\theta(t)\|_{L^{r}}&\leq \left(\|\theta_{0}\|_{L^{r}}^{\frac{r}{2}}
+\left(\frac{2(r-2)}{\widetilde{c}}\right)^{\frac{r-2}{2}}
C(t)^{\frac{r}{2}}\right)^{\frac{2}{r}}\nonumber \\ &\leq
\|\theta_{0}\|_{L^{r}} +\left(\frac{2(r-2)}{\widetilde{c}}\right)^{\frac{r-2}{r}}
C(t).
\end{align}
We notice that
\begin{align}
\left(\frac{2(r-2)}{\widetilde{c}}\right)^{\frac{r-2}{r}}
&=\left(\frac{2}{\widetilde{c}}\right)^{\frac{r-2}{r}}\left(r-2\right)^{\frac{r-2}{r}}
\nonumber\\ &=\left(\frac{2}{\widetilde{c}}\right)^{\frac{r-2}{r}}
\left(r-2\right)^{-\frac{2}{r}} \frac{r-2}{r}\, r,\nonumber
\end{align}
which leads to
\begin{align}\label{xdgqxz11}
\lim_{r\rightarrow +\infty}\frac{\left(\frac{2(r-2)}{\widetilde{c}}\right)^{\frac{r-2}{r}}}{r}
=\lim_{r\rightarrow +\infty}\left(\left(\frac{2}{\widetilde{c}}\right)^{\frac{r-2}{r}}
\left(r-2\right)^{-\frac{2}{r}} \frac{r-2}{r}\right)=\frac{2}{\widetilde{c}}.
\end{align}
Consequently, we conclude from \eqref{xdgqxz11} that there exists a positive constant $C=C(\widetilde{c})$ is independent of $r$ such that
$$\left(\frac{2(r-2)}{\widetilde{c}}\right)^{\frac{r-2}{r}}\leq C \,r,$$
which along with \eqref{xfyped26} implies
\begin{align}
\|\theta(t)\|_{L^{r}}&\leq \|\theta_{0}\|_{L^{r}} +C(t)\left(\frac{2(r-2)}{\widetilde{c}}\right)^{\frac{r-2}{r}}
\nonumber \\ &\leq
\|\theta_{0}\|_{L^{r}} +
C\,r.\nonumber
\end{align}
As a result, (\ref{yydkl028}) follows by dividing both sides of the above inequality by $r$. We thus end the proof of Proposition \ref{Gsds1}.
\end{proof}

\vskip .1in
Next we are going to show the $L^{r}$-bound of the vorticity $\omega$ and $B_{r,1}^{\sigma}$-bound of $\theta$ with $\sigma$ and $r$ satisfying \eqref{vnkhgfd66} below.
\begin{Pros}\label{Pro5add1}
Assume $(u_{0},\,\theta_{0})$ satisfies the assumptions stated in
Theorem \ref{Th1}, then the corresponding solution $(u,\theta)$ satisfies
\begin{eqnarray}\label{tt5add01}
\|\theta(t)\|_{B_{r,1}^{\sigma}} +\|\omega(t)\|_{L^{2}}
+\|\omega(t)\|_{L^{r}} +\int_{0}^{t}{\|\theta(\tau)\|_{B_{r,1}^{\sigma+1}} \,d\tau}\leq C(t,\,u_{0},\,\theta_{0}),
\end{eqnarray}
where $\sigma$ and $r$ satisfy
\begin{align}\label{vnkhgfd66}
\frac{2}{r}-1<\sigma<0,\quad 2<r<\infty.
\end{align}
\end{Pros}

\begin{rem}
After checking the proof below, Proposition \ref{Pro5add1} still holds true when $\sigma$ and $r$ satisfy $$\frac{2}{r}-1<\sigma<1,\quad 2<r<\infty$$
provided in addition $\theta_{0}\in B_{r,1}^{\sigma}(\mathbb{R}^{2})$.
\end{rem}

\begin{proof}
We apply $\Delta_{j}$ to $(\ref{BEBou})_{2}$ to obtain
\begin{align}\label{tt5add02}
\partial_{t}\Delta_{j}\theta+(u \cdot \nabla) \Delta_{j}\theta+\Delta_{j}
\Lambda \theta=-[\Delta_{j},\,u \cdot \nabla]\theta+\Delta_{j}u_{2}.
\end{align}
Multiplying (\ref{tt5add02}) by $|\Delta_{j}\theta|^{r-2}\Delta_{j}\theta$, it follows that
\begin{eqnarray}\label{tt5add03}
\frac{1}{r}\frac{d}{dt}
\|\Delta_{j}\theta(t)\|_{L^{r}}^{r} +c2^{
j}\|\Delta_{j}\theta\|_{L^{r}}^{r}\leq \|[\Delta_{j},\,u \cdot
\nabla]\theta\|_{L^{r}}\|\Delta_{j}\theta\|_{L^{r}}^{r-1}
+\|\Delta_{j}u\|_{L^{r}}\|\Delta_{j}\theta\|_{L^{r}}^{r-1},
\end{eqnarray}
where the lower bound was applied (see \cite{CMZ})
$$\int_{\mathbb{R}^{2}}
(\Lambda \Delta_{j}\theta)|\Delta_{j}\theta|^{
r-2}\Delta_{j}\theta\,dx\geq c2^{
j}\|\Delta_{j}\theta\|_{L^{r}}^{r}$$
for some absolute constant $c>0$.
Consequently, we obtain from \eqref{tt5add03} that
\begin{align}\label{tt5add04}
\frac{d}{dt}
\|\Delta_{j}\theta(t)\|_{L^{r}}  +c2^{j}\|\Delta_{j}\theta\|_{L^{r}}
\leq \|[\Delta_{j},\,u \cdot \nabla]\theta\|_{L^{r}}
+\|\Delta_{j}u\|_{L^{r}}.
\end{align}
Multiplying each side of \eqref{tt5add04} by $2^{\sigma j}$ and taking sup with respect to $j$, we derive
\begin{align}\label{tt5add05}
\frac{d}{dt}
\|\theta(t)\|_{B_{r,1}^{\sigma}}+c\|\theta\|_{B_{r,1}^{\sigma+1}}
\leq& C\sum_{j\geq-1}2^{\sigma j}\|[\Delta_{j},\,u \cdot \nabla]\theta\|_{L^{r}}
+C\sum_{j\geq-1}2^{\sigma j}\|\Delta_{j}u\|_{L^{r}}.
\end{align}
Due to $\sigma<0$, it happens that
\begin{align}\label{tt5add06}
C\sum_{j\geq-1}2^{\sigma j}\|\Delta_{j}u\|_{L^{r}}
&\leq C\sum_{j\geq-1}2^{\sigma j}\|u\|_{L^{r}} \nonumber\\
&\leq C(\|u\|_{L^{2}}+\|\omega\|_{L^{2}})
\nonumber\\
&\leq C(1+\|\omega\|_{L^{2}}).
\end{align}
Putting \eqref{tt5add06} into \eqref{tt5add05} implies
\begin{align}\label{tt5add08}
\frac{d}{dt}
\|\theta(t)\|_{B_{r,1}^{\sigma}}+c\|\theta\|_{B_{r,1}^{\sigma+1}}
\leq&  C(1+\|\omega\|_{L^{2}})+C\sum_{j\geq-1}2^{\sigma j}\|[\Delta_{j},\,u \cdot \nabla]\theta\|_{L^{r}}
\end{align}
We next claim
\begin{align}\label{bmhgfxa86}
C\sum_{j\geq-1}2^{\sigma j}\|[\Delta_{j},\,u \cdot \nabla]\theta\|_{L^{r}} \leq C\|\nabla u\|_{L^{r}}
\|\theta\|_{B_{\infty,1}^{\sigma}},
\end{align}
where $-1<\sigma<1$. Invoking the Bony decomposition, the commutator of \eqref{bmhgfxa86} can be written as
\begin{align}\label{tt5add09}
[\Delta_{j},\,u\cdot\nabla]\theta= & \sum_{|j-q|\leq
4}[\Delta_{j},\,S_{q-1}u\cdot\nabla]\Delta_{q}\theta+\sum_{|j-q|\leq
4}[\Delta_{j},\,\Delta_{q}u\cdot\nabla]S_{q-1}\theta\nonumber\\& +\sum_{q-j
\geq
-4}[\Delta_{j},\,\Delta_{q}u\cdot\nabla]\widetilde{\Delta}_{q}\theta\nonumber\\ \triangleq&
M_{1}^{j}+M_{2}^{j}+M_{3}^{j}.
\end{align}
Let us recall the fact (\cite[Lemma 3.2]{HK4})
\begin{align}\label{tt5add10}
\|h\ast(fg)-f(h\ast g)\|_{L^{p}}\leq \|xh\|_{L^{1}}\|\nabla f\|_{L^{p}}\|g\|_{L^{\infty}},
\end{align}
where $\ast$ stands for the convolution symbol.
The straightforward computations combined with Lemma \ref{vfgty8xc} and \eqref{tt5add10} ensure that
\begin{align}\label{tt5add11}
 C\sum_{j\geq-1}2^{\sigma j}\|M_{1}^{j}\|_{L^{r}} &= C\sum_{j\geq-1}2^{\sigma j}\Big\|\sum_{|j-q|\leq 4}\big(h_{q}\ast(S_{j-1}u\cdot\nabla\Delta_{j}\theta)-
S_{j-1}u\cdot(h_{q}\ast\nabla\Delta_{j}\theta)\big)\big\|_{L^{r}} \nonumber\\
 &\leq C \sum_{j\geq-1}2^{\sigma j}\sum_{|j-q|\leq 4}\|x h_{q}\|_{L^{1}} \|\nabla S_{j-1}u\|_{L^{r}} \|\nabla\Delta_{j}\theta\|_{L^{\infty}}
\nonumber\\ &\leq C\sum_{j\geq-1}2^{\sigma j}\sum_{|j-q|\leq 4}2^{j-q} \|\nabla u\|_{L^{r}}  \|\Delta_{j}\theta\|_{L^{\infty}} \nonumber\\ &\leq C\|\nabla u\|_{L^{r}}  \sum_{j\geq-1}2^{\sigma j} \|\Delta_{j}\theta\|_{L^{\infty}}
\nonumber\\ &\leq C\|\nabla u\|_{L^{r}}
\|\theta\|_{B_{\infty,1}^{\sigma}},
\end{align}
where $h_{q}(x)=2^{2q}\check{\varphi}(2^{q}x)$.
Similarly, we also end up with
\begin{align}\label{tt5add12}
C\sum_{j\geq-1}2^{\sigma j}\|M_{2}^{j}\|_{L^{r}}  \leq&C\sum_{j\geq-1}2^{\sigma j}\sum_{|j-q|\leq 4}2^{-q}
\|\Delta_{j}\nabla u\|_{L^{r}} \|\nabla S_{j-1}\theta\|_{L^{\infty}}
\nonumber\\ \leq&C\sum_{j\geq-1}2^{\sigma j}\sum_{|j-q|\leq 4}2^{-q}
\|\nabla u\|_{L^{r}}\|\nabla S_{j-1}\theta\|_{L^{\infty}}
\nonumber\\ \leq&C\|\nabla u\|_{L^{r}}\sum_{j\geq-1}2^{(\sigma-1)j}
\|\nabla S_{j-1}\theta\|_{L^{\infty}}
\nonumber\\ \leq&C\|\nabla u\|_{L^{r}} \sum_{j\geq-1}2^{(\sigma-1)j}\sum_{k\leq j-2}
\|\nabla \Delta_{k}\theta\|_{L^{\infty}}
\nonumber\\ \leq&C\|\nabla u\|_{L^{r}} \sum_{j\geq-1}2^{(\sigma-1)j}\sum_{k\leq j-2}2^{-k(\sigma-1)}
2^{k\sigma}\|\Delta_{k}\theta\|_{L^{\infty}}
\nonumber\\ \leq&C\|\nabla u\|_{L^{r}} \sum_{j\geq-1}\sum_{k\leq j-2}2^{(\sigma-1)(j-k)}
2^{k\sigma}\|\Delta_{k}\theta\|_{L^{\infty}}
\nonumber\\ \leq& C\|\nabla u\|_{L^{r}}
\|\theta\|_{B_{\infty,1}^{\sigma}},
\end{align}
where we have used the Young inequality for series due to $\sigma<1$.
We further split the term $M_{3}^{j}$ into two parts
\begin{align}
M_{3}^{j} =\sum_{q-j \geq -4,\,q\geq0}[\Delta_{j},\,\Delta_{q}u\cdot\nabla]\widetilde{\Delta}_{q}\theta+
[\Delta_{j},\,\Delta_{-1}u\cdot\nabla]\widetilde{\Delta}_{-1}\theta\triangleq M_{31}^{j}+M_{32}^{j}.\nonumber
\end{align}
Thanks to $\nabla\cdot u=0$, the term $M_{31}^{j}$ can be bounded without the use of  commutator structure. More precisely, a straightforward computations yield
\begin{align}\label{tt5add13}
&C\sum_{j\geq-1}2^{\sigma j}\|M_{31}^{j}\|_{L^{r}} \nonumber\\&=C\sum_{j\geq-1}2^{\sigma j}
\left\|\sum_{q-j \geq -4,\,q\geq0}\Big\{\Delta_{j}\nabla\cdot(\Delta_{q}u
\widetilde{\Delta}_{q}\theta)-\Delta_{q}u\Delta_{j}\nabla\cdot
\widetilde{\Delta}_{q}\theta\Big\}\right\| _{L^{r}}
\nonumber\\ &\leq
C\sum_{j\geq-1}2^{\sigma j}
\sum_{q-j \geq -4,\,q\geq0}\left(\|\Delta_{j}\nabla\cdot(\Delta_{q}u
\widetilde{\Delta}_{q}\theta)\|_{L^{r}} +\| \Delta_{q}u\Delta_{j}\nabla\cdot
\widetilde{\Delta}_{q}\theta\|_{L^{r}}\right)
\nonumber\\ &\leq
C\sum_{j\geq-1}2^{\sigma j}
\sum_{q-j \geq -4,\,q\geq0}\left(2^{j} \|\Delta_{q}u
\widetilde{\Delta}_{q}\theta \|_{L^{r}}+ \| \Delta_{q}u \|_{L^{r}}\|\Delta_{j}\nabla\cdot
\widetilde{\Delta}_{q}\theta \| _{L^{\infty}}\right)
\nonumber\\ &\leq C\sum_{j\geq-1}2^{\sigma j}\sum_{q-j \geq -4,\,q\geq0}2^{j}
\|\Delta_{q} u\|_{L^{r}}\|\widetilde{\Delta}_{q}\theta\|_{L^{\infty}}\nonumber\\
 &\leq C\sum_{j\geq-1}2^{\sigma j}\sum_{q-j \geq -4,\,q\geq0}2^{j-q}
\|\Delta_{q}\nabla u\|_{L^{r}}\|\widetilde{\Delta}_{q}\theta\|_{L^{\infty}}
\nonumber\\
 &\leq C\sum_{j\geq-1}2^{\sigma j}\sum_{q-j \geq -4}2^{j-q}2^{-\sigma q}
\| \nabla u\|_{L^{r}}2^{\sigma q}\|\widetilde{\Delta}_{q}\theta\|_{L^{\infty}}
\nonumber\\
 &\leq C\| \nabla u\|_{L^{r}} \sum_{j\geq-1} \sum_{q-j \geq -4} 2^{(\sigma+1)(j-q)}
2^{\sigma q}\|\widetilde{\Delta}_{q}\theta\|_{L^{\infty}}
\nonumber\\ &\leq  C\|\nabla u\|_{L^{r}}
\|\theta\|_{B_{\infty,1}^{\sigma}},
\end{align}
where we have used the Young inequality for series due to $\sigma>-1$.
Finally, Lemma \ref{vfgty8xc} and \eqref{tt5add10} allow us to derive
\begin{align}\label{tt5add14}
C\sum_{j\geq-1}2^{\sigma j}\|M_{32}^{j}\|_{L^{r}}  &=C\sum_{-1\leq j\leq 2 }2^{\sigma j}\|M_{32}^{j}\|_{L^{r}}
\nonumber\\ &\leq C \|\Delta_{-1}\nabla u\|_{L^{r}}
\|\widetilde{\Delta}_{-1} \nabla\theta\|_{L^{\infty}} \nonumber\\ & \leq C\|\nabla u\|_{L^{r}}
\|\theta\|_{B_{\infty,1}^{\sigma}}.
\end{align}
Putting \eqref{tt5add11}, \eqref{tt5add12},  \eqref{tt5add13} and  \eqref{tt5add14} into \eqref{tt5add09} gives the desired bound \eqref{bmhgfxa86}. Now inserting \eqref{bmhgfxa86} into \eqref{tt5add08} implies
\begin{align}\label{tt5add15}
\frac{d}{dt}
\|\theta(t)\|_{B_{r,1}^{\sigma}}+c\|\theta\|_{B_{r,1}^{\sigma+1}}
\leq&  C(1+\|\omega\|_{L^{2}})+ C\|\nabla u\|_{L^{r}}
\|\theta\|_{B_{\infty,1}^{\sigma}}\nonumber\\
\leq& C(1+\|\omega\|_{L^{2}})+C\|\nabla u\|_{L^{r}}
\|\theta\|_{B_{\rho,1}^{\sigma+\frac{2}{\rho}}}
\nonumber\\
\leq& C(1+\|\omega\|_{L^{2}})+C\|\omega\|_{L^{r}}
\|\theta\|_{L^{\rho}}
\nonumber\\
\leq&
 C(1+\|\omega\|_{L^{2}})+C\|\theta\|_{L^{\rho}} \|\omega\|_{L^{r}}
\nonumber\\
\leq&
 C(1+\|\omega\|_{L^{2}})+C\|\omega\|_{L^{r}}
 \nonumber\\
\leq&
 C(1+\|\omega\|_{L^{2}}+\|\omega\|_{L^{r}}),
\end{align}
where we have used the simple embedding $L^{\rho}(\mathbb{R}^2)\hookrightarrow B_{\rho,1}^{\sigma+\frac{2}{\rho}}(\mathbb{R}^2)$ with $\rho\in (-\frac{2}{\sigma}, \infty)$ and $\sigma<0$.
It thus follows from \eqref{tt5add15} that
\begin{align}\label{tt5add16}
\frac{d}{dt}
\|\theta(t)\|_{B_{r,1}^{\sigma}}+c\|\theta\|_{B_{r,1}^{\sigma+1}}
 \leq C(1+\|\omega\|_{L^{2}}+\|\omega\|_{L^{r}}).
\end{align}
In order to close \eqref{tt5add16}, we appeal to the vorticity $\omega$ equation
\begin{eqnarray}\label{BEB002}
\partial_{t}\omega+(u\cdot\nabla)\omega=\partial_{1}\theta.
\end{eqnarray}
However, due to $\sigma<0$, the terms $\|\omega\|_{L^{2}}$ and $\|\omega\|_{L^{r}}$ at the right hand side of \eqref{tt5add16} can not be closed effectively by the vorticity equation \eqref{BEB002} itself. To overcome this difficulty, one can resort to take advantage of the elegant method introduced by Hmidi, Keraani and Rousset \cite{HK3,HK4} to deal with the 2D critical Boussinesq equations. Applying $\mathcal {R}_{1}$ to $(\ref{BEBou})_{2}$, one has
\begin{eqnarray}\label{vbndf56}
\partial_{t}\mathcal {R}_{1}\theta+(u \cdot \nabla) \mathcal {R}_{1}\theta+\Lambda \mathcal {R}_{1}\theta=\mathcal {R}_{1}u_{2}-[\mathcal {R}_{1},u \cdot \nabla]\theta.
\end{eqnarray}
Concerning $(\ref{vbndf56})$ and \eqref{BEB002}, it is not hard to check that the following combined quantity
$$\Gamma\triangleq\omega+\mathcal {R}_{1} \theta,$$
obeys the equation
\begin{eqnarray}\label{fgt504}
\partial_{t}\Gamma+(u\cdot\nabla)
\Gamma=\mathcal {R}_{1}u_{2}-[\mathcal
{R}_{1},\,u\cdot\nabla]\theta.
\end{eqnarray}
In fact, the equation \eqref{fgt504} plays an important role in closing  \eqref{tt5add16}. To this end, multiplying \eqref{fgt504} by $|\Gamma|^{r-2}\Gamma$ and using \eqref{sdfvh80}, we find
\begin{align}\label{flgdaew11}
\frac{1}{r}\frac{d}{dt}\|\Gamma(t)\|_{L^{r}}^{r}
 &=\int_{\mathbb{R}^{2}}\mathcal {R}_{1}u_{2}
|\Gamma|^{r-2}\Gamma\,dx-\int_{\mathbb{R}^{2}}[\mathcal
{R}_{1},\,u\cdot\nabla]\theta
|\Gamma|^{r-2}\Gamma\,dx
\nonumber\\
& \leq \|\mathcal {R}_{1}u_{2}\|_{L^{r}}\|\Gamma \|_{L^{r}}^{r-1}+\|[\mathcal
{R}_{1},\,u\cdot\nabla]\theta\|_{L^{r}}\|\Gamma\|_{L^{r}}^{r-1}
\nonumber\\
& \leq C\|u\|_{L^{r}}\|\Gamma(t)\|_{L^{r}}^{r-1}+C\|\nabla u \|_{L^{r}}\|\theta\|_{L^{\infty}}\|\Gamma\|_{L^{r}}^{r-1}
\nonumber\\
& \leq C\|u\|_{L^{2}}^{\frac{r}{2r-2}}\|\omega\|_{L^{r}}^{\frac{r-2}{2r-2}}
\|\Gamma\|_{L^{r}}^{r-1}
+C\|\omega\|_{L^{r}}\|\theta\|_{L^{\infty}}\|\Gamma\|_{L^{r}}^{r-1}
\nonumber\\
& \leq C\|u\|_{L^{2}}^{\frac{r}{2r-2}}(\|\Gamma\|_{L^{r}}+\|\mathcal {R}_{1}\theta\|_{L^{r}})^{\frac{r-2}{2r-2}}
\|\Gamma\|_{L^{r}}^{r-1}
\nonumber\\
&\quad+C(\|\Gamma\|_{L^{r}}+\|\mathcal {R}_{1}\theta\|_{L^{r}})\|\theta\|_{L^{\infty}}\|\Gamma\|_{L^{r}}^{r-1}
\nonumber\\
& \leq C\|u\|_{L^{2}}^{\frac{r}{2r-2}}(\|\Gamma\|_{L^{r}}+\|\theta\|_{L^{r}})
^{\frac{r-2}{2r-2}}
\|\Gamma\|_{L^{r}}^{r-1}
\nonumber\\&\quad+C(\|\Gamma\|_{L^{r}}+\|\theta\|_{L^{r}})\|\theta\|_{L^{\infty}}
\|\Gamma\|_{L^{r}}^{r-1},
\end{align}
which leads to
\begin{align}\label{fzhpet17}
 \frac{d}{dt}\|\Gamma(t)\|_{L^{r}}
\leq C(1+\|\theta\|_{L^{\infty}})(1+\|\Gamma\|_{L^{r}}).
\end{align}
In particular, taking $r=2$ in \eqref{flgdaew11}, we also have
\begin{align}\label{vbjlwq6}
 \frac{d}{dt}\|\Gamma(t)\|_{L^{2}}
\leq C(1+\|\theta\|_{L^{\infty}})(1+\|\Gamma\|_{L^{2}})
\end{align}
Let us also emphasize that after checking the proof of \cite[Proposition 5.2]{HK4}, we may find
\begin{align} \label{fngkp897}
 \frac{d}{dt}\|\Gamma(t)\|_{L^{r}}
\leq C(1+\|\theta\|_{B_{\infty,2}^{0}})(1+\|\Gamma\|_{L^{r}}).
\end{align}
As space $L^{\infty}$ and Besov space $B_{\infty,2}^{0}$ do not contain each other, we are not able to distinguish which is better among \eqref{fzhpet17} and \eqref{fngkp897}. As stated in the introduction part, \eqref{fngkp897} is insufficient to continue with the subsequent estimates. In fact, \eqref{fzhpet17} plays a crucial role in proving our theorem.
Summing up \eqref{tt5add16}, \eqref{fzhpet17} and \eqref{vbjlwq6}, we thus have
\begin{align}\label{fzhpet18}
 &\frac{d}{dt}(\|\theta(t)\|_{B_{r,1}^{\sigma}}+\|\Gamma(t)\|_{L^{2}}
 +\|\Gamma(t)\|_{L^{r}})
+\|\theta\|_{B_{r,1}^{\sigma+1}}
\nonumber\\&\leq C(1+\|\omega\|_{L^{2}}+\|\omega\|_{L^{r}})+ C(1+\|\theta\|_{L^{\infty}})(1+\|\Gamma\|_{L^{2}}
 +\|\Gamma\|_{L^{r}})
\nonumber\\&\leq
C(1+\|\Gamma\|_{L^{2}}+\|\mathcal {R}_{1}\theta\|_{L^{2}}
+\|\Gamma\|_{L^{r}}+\|\mathcal {R}_{1}\theta\|_{L^{r}})\nonumber\\&\quad+ C(1+\|\theta\|_{L^{\infty}})(1+\|\Gamma\|_{L^{2}}
 +\|\Gamma\|_{L^{r}})
\nonumber\\&\leq C(1+\|\theta\|_{L^{\infty}})\left(1+\|\theta\|_{B_{r,1}^{\sigma}}+\|\Gamma\|_{L^{2}}
 +\|\Gamma\|_{L^{r}}\right).
\end{align}
We denote
$$A(t)\triangleq 1+ \|\theta(t)\|_{B_{r,1}^{\sigma}} +\|\Gamma(t)\|_{L^{2}}
 +\|\Gamma(t)\|_{L^{r}},\qquad B(t)\triangleq \|\theta(t)\|_{B_{r,1}^{\sigma+1}}.$$
It thus follows from \eqref{fzhpet18} that
\begin{equation}\label{fzhpet19}
\frac{d}{dt}A(t)+B(t)\leq CA(t)+C\|\theta(t)\|_{L^{\infty}}A(t).
\end{equation}
To control $\|\theta\|_{L^{\infty}}$ via \eqref{yydkl028}, we need the following logarithmic Sobolev interpolation inequality (see \cite[Proposition 5.2]{ACW11})
\begin{eqnarray}\label{LOG2}
\|\theta\|_{L^{\infty}}
\leq C\left(1+\sup_{m\geq2}\frac{\|\theta\|_{L^{m}}}{m}
\ln\big(e+\|\theta\|_{B_{r,1}^{\sigma+1}}\big)\right),\quad \sigma>\frac{2}{r}-1.
\end{eqnarray}
Applying \eqref{LOG2} to \eqref{fzhpet19} shows
\begin{equation}\label{fzhpet20}
\frac{d}{dt}A(t)+B(t)\leq C A(t)+Cn(t)
\ln\big(e+B(t)\big)A(t),
\end{equation}
where $n(t)$ is given by
$$n(t)\triangleq \sup_{m\geq2}\frac{\|\theta(t)\|_{L^{m}}}{m}.$$
Due to \eqref{yydkl028}, we have that $n(t)$ belongs to $L_{loc}^{1}(\mathbb{R}_{+})$. Moreover, it is not hard to check that $n(t)$ satisfies \eqref{ympery8}, namely,
\begin{align}
 n(t)
&\leq \|\theta\|_{L^{2}}+\|\theta\|_{L^{\infty}}\nonumber\\ &\leq \|\theta\|_{L^{2}}+C\|\theta\|_{L^{2}}^{1-\frac{1}{\sigma+2-\frac{2}{r}}}
\|\theta\|_{B_{r,1}^{\sigma+1}}^{\frac{1}{\sigma+2-\frac{2}{r}}}
\nonumber\\ &\leq C+C
\|\theta\|_{B_{r,1}^{\sigma+1}}^{\frac{1}{\sigma+2-\frac{2}{r}}}
\nonumber\\ &\leq C
\big(1+ B(t)\big)^{\frac{1}{\sigma+2-\frac{2}{r}}}\nonumber
\end{align}
with $\frac{1}{\sigma+2-\frac{2}{r}}<1$.
We now apply Lemma \ref{Lem01khj} to \eqref{fzhpet20} to conclude
\begin{equation}
A(t)+\int_{0}^{t}{B(\tau)\,d\tau}\leq C(t),\nonumber
\end{equation}
which is nothing but
$$\|\theta(t)\|_{B_{r,1}^{\sigma}} +\|\Gamma(t)\|_{L^{2}}
+\|\Gamma(t)\|_{L^{r}}
+\int_{0}^{t}{\|\theta(\tau)\|_{B_{r,1}^{\sigma+1}} \,d\tau}\leq C(t,\,u_{0},\,\theta_{0}).$$
Due to the relation $\Gamma=\omega+\mathcal {R}_{1} \theta$ and \eqref{yydkl028}, we readily get
$$\|\theta(t)\|_{B_{r,1}^{\sigma}} +\|\omega(t)\|_{L^{2}}
+\|\omega(t)\|_{L^{r}}
+\int_{0}^{t}{\|\theta(\tau)\|_{B_{r,1}^{\sigma+1}} \,d\tau}\leq C(t,\,u_{0},\,\theta_{0}).$$
Therefore, we end up with the proof of Proposition \ref{Pro5add1}.
\end{proof}

\vskip .1in
With the global bounds in Proposition \ref{Pro5add1} at our disposal, we are ready to show the key $L^{\infty}$-boundedness of the vorticity.
\begin{Pros} \label{Gsds3}
Assume $(u_{0},\,\theta_{0})$ satisfies the assumptions stated in
Theorem \ref{Th1}, then the corresponding solution $(u,\theta)$ admits the following global bound
\begin{eqnarray}\label{yyfvxaq226}
\|\omega(t)\|_{L^{\infty}}+\|\mathcal {R}_{1}\theta(t)\|_{L^{\infty}} \leq C(t,\,u_{0},\,\theta_{0}).
\end{eqnarray}
\end{Pros}
\begin{proof}
Multiplying \eqref{fgt504} by $|\Gamma|^{r-2}\Gamma$, we shall have
\begin{align}
\frac{1}{r}\frac{d}{dt}\|\Gamma(t)\|_{L^{r}}^{r}
&=\int_{\mathbb{R}^{2}}
{\mathcal {R}_{1}u_{2}\,\,(|\Gamma|^{r-2}\Gamma)\,dx}-\int_{\mathbb{R}^{2}}
{[\mathcal
{R}_{1},\,u\cdot\nabla]\theta\,\,(|\Gamma|^{r-2}\Gamma)\,dx}\nonumber\\&
\leq  \|\mathcal {R}_{1}u_{2}\|_{L^{r}}\|\Gamma\|_{L^{r}}^{r-1}+ \|[\mathcal
{R}_{1},\,u\cdot\nabla]\theta\|_{L^{r}}\|\Gamma\|_{L^{r}}^{r-1},\nonumber
\end{align}
which gives
\begin{align}
 \frac{d}{dt}\|\Gamma(t)\|_{L^{r}}
\leq  \|\mathcal {R}_{1}u_{2}\|_{L^{r}}+ \|[\mathcal
{R}_{1},\,u\cdot\nabla]\theta\|_{L^{r}}.\nonumber
\end{align}
Letting $r\rightarrow\infty$ shows
\begin{align}
 \frac{d}{dt}\|\Gamma(t)\|_{L^{\infty}}
\leq  \|\mathcal {R}_{1}u_{2}\|_{L^{\infty}}+ \|[\mathcal
{R}_{1},\,u\cdot\nabla]\theta\|_{L^{\infty}},\nonumber
\end{align}
which yields the maximum principle
\begin{align}\label{yyfvxaq227}
\|\Gamma(t)\|_{L^{\infty}}\leq \|\Gamma_{0}\|_{L^{\infty}}+\int_{0}^{t}{
\|\mathcal {R}_{1}u_{2}(\tau)\|_{L^{\infty}}\,d\tau} +\int_{0}^{t}{
\|[\mathcal
{R}_{1},\,u\cdot\nabla]\theta(\tau)\|_{L^{\infty}}\,d\tau}.
\end{align}
Keeping in mind the fact
$$\int_{\mathbb{R}^{2}}
(\Lambda f)|f|^{r-2}f\,dx\geq 0,$$
it also follows from \eqref{vbndf56} that
\begin{align}\label{yyfvxaq228}
\|\mathcal {R}_{1}\theta(t)\|_{L^{\infty}}\leq \|\mathcal {R}_{1}\theta_{0}\|_{L^{\infty}}+\int_{0}^{t}{
\|\mathcal {R}_{1}u_{2}(\tau)\|_{L^{\infty}}\,d\tau} +\int_{0}^{t}{
\|[\mathcal
{R}_{1},\,u\cdot\nabla]\theta(\tau)\|_{L^{\infty}}\,d\tau}.
\end{align}
Putting together \eqref{yyfvxaq227} and \eqref{yyfvxaq228}, it implies
\begin{align}\label{yyfvxaq229}
\|\Gamma(t)\|_{L^{\infty}}+\|\mathcal {R}_{1}\theta(t)\|_{L^{\infty}}\leq & \|\Gamma_{0}\|_{L^{\infty}}+\|\mathcal {R}_{1}\theta_{0}\|_{L^{\infty}}+2\int_{0}^{t}{
\|\mathcal {R}_{1}u_{2}(\tau)\|_{L^{\infty}}\,d\tau} \nonumber\\&+2\int_{0}^{t}{
\|[\mathcal
{R}_{1},\,u\cdot\nabla]\theta(\tau)\|_{L^{\infty}}\,d\tau}
\nonumber\\ \leq &
\|\omega_{0}\|_{L^{\infty}}+2\|\mathcal {R}_{1}\theta_{0}\|_{L^{\infty}}+2\int_{0}^{t}{
\|\mathcal {R}_{1}u_{2}(\tau)\|_{L^{\infty}}\,d\tau} \nonumber\\&+2\int_{0}^{t}{
\|[\mathcal
{R}_{1},\,u\cdot\nabla]\theta(\tau)\|_{L^{\infty}}\,d\tau}
\nonumber\\ \leq & C+2\int_{0}^{t}{
\|\mathcal {R}_{1}u_{2}(\tau)\|_{L^{\infty}}\,d\tau}
+2\int_{0}^{t}{
\|[\mathcal
{R}_{1},\,u\cdot\nabla]\theta(\tau)\|_{L^{\infty}}\,d\tau}\nonumber\\ \leq & C+C\int_{0}^{t}{(\|u(\tau)\|_{L^{2}}+\|\nabla u(\tau)\|_{L^{4}})\,d\tau}
\nonumber\\&+2\int_{0}^{t}{
\|[\mathcal
{R}_{1},\,u\cdot\nabla]\theta(\tau)\|_{L^{\infty}}\,d\tau}\nonumber\\ \leq & C
+2\int_{0}^{t}{
\|[\mathcal
{R}_{1},\,u\cdot\nabla]\theta(\tau)\|_{L^{\infty}}\,d\tau}
\nonumber\\ \leq & C
+C\int_{0}^{t}{
\|[\mathcal
{R}_{1},\,u\cdot\nabla]\theta(\tau)\|_{B_{\infty,1}^{0}}\,d\tau}
\nonumber\\ \leq & C
+C\int_{0}^{t}{
(\|\omega\|_{L^{\infty}}+\|\omega\|_{L^{2}})
(\|\theta\|_{B_{\infty,1}^{\epsilon}}
+\|\theta\|_{L^{2}})(\tau)\,d\tau},
\end{align}
where in the last line we have used the commutator estimate due to \cite[Theorem 3.3]{HK4}
$$\|[\mathcal
{R}_{1},\,u\cdot\nabla]\theta\|_{B_{\infty,1}^{0}}\leq C(\|\omega\|_{L^{\infty}}+\|\omega\|_{L^{2}})(\|\theta\|_{B_{\infty,1}^{\epsilon}}
+\|\theta\|_{L^{2}}),\qquad \forall\,\epsilon>0.$$
We deduce from \eqref{yyfvxaq229} that
\begin{align}\label{gdyhkp866}
\|\Gamma(t)\|_{L^{\infty}}+\|\mathcal {R}_{1}\theta(t)\|_{L^{\infty}}\leq & C
+C\int_{0}^{t}{
(\|\omega\|_{L^{\infty}}+\|\omega\|_{L^{2}})(\|\theta\|_{B_{\infty,1}^{\epsilon}}
+\|\theta\|_{L^{2}})(\tau)\,d\tau},
\end{align}
which along with the fact $\|\omega\|_{L^{\infty}}\leq\|\Gamma \|_{L^{\infty}}+\|\mathcal {R}_{1}\theta \|_{L^{\infty}}$ yields
\begin{align}\label{yyfvxaq230}
\|\omega(t)\|_{L^{\infty}}\leq & C
+C\int_{0}^{t}{
(\|\omega(\tau)\|_{L^{\infty}}+\|\omega(\tau)\|_{L^{2}})(\|\theta(\tau)\|_{B_{\infty,1}^{\epsilon}}
+\|\theta(\tau)\|_{L^{2}})\,d\tau}\nonumber\\ \leq & C
+C\int_{0}^{t}{
(1+\|\omega(\tau)\|_{L^{\infty}})(\|\theta(\tau)\|_{B_{\infty,1}^{\epsilon}}
+\|\theta(\tau)\|_{L^{2}})\,d\tau}.
\end{align}
Fixing $\epsilon\in (0,\sigma+1-\frac{2}{r})$ and using \eqref{tt5add01}, we have that
\begin{align}\label{vxvtaq828}
\int_{0}^{t}{(\|\theta(\tau)\|_{B_{\infty,1}^{\epsilon}}
+\|\theta(\tau)\|_{L^{2}})\,d\tau}&\leq C\int_{0}^{t}{(\|\theta(\tau)\|_{B_{r,1}^{\sigma+1}}
+\|\theta(\tau)\|_{L^{2}})\,d\tau}\nonumber\\&\leq C(t,\,u_{0},\,\theta_{0}).
\end{align}
Keeping in mind this bound and applying the Gronwall inequality to \eqref{yyfvxaq230}, we then get
$$\|\omega(t)\|_{L^{\infty}}\leq C(t,\,u_{0},\,\theta_{0}).$$
This together with \eqref{gdyhkp866} and \eqref{vxvtaq828} yield
$$\|\mathcal {R}_{1}\theta(t)\|_{L^{\infty}}\leq C(t,\,u_{0},\,\theta_{0}).$$
In consequence, we finish the proof of Proposition \ref{Gsds3}.
\end{proof}

\subsection{Existence.}
To show the existence of the solution, we first smooth out the initial data to get the following approximate system
\begin{equation}\label{APPBEBou}
\left\{\aligned
&\partial_{t}u^{(n)}+(u^{(n)} \cdot \nabla) u^{(n)}+\nabla p^{(n)}=\theta^{(n)} e_{2},\,\,\,\,\,\,\,x\in \mathbb{R}^{2},\,\,t>0, \\
&\partial_{t}\theta^{(n)}+(u^{(n)} \cdot \nabla) \theta^{(n)}+\Lambda \theta^{(n)}=u_{2}^{(n)}, \\
&\nabla\cdot u^{(n)}=0,\\
&u^{(n)}(x, 0)=S_{n}u_{0}(x),  \quad \theta^{(n)}(x,0)=S_{n}\theta_{0}(x),
\endaligned\right.
\end{equation}
where $S_{n}$ is a smooth low-frequency cut-off operator. Obviously, $(S_{n}u_{0}, S_{n}\theta_{0})$ is bounded in the space given in the statement of Theorem \ref{Th1}. Moreover, $S_{n}u_{0}$ and $ S_{n}\theta_{0}$ belong to all $H^{s}(\mathbb{R}^n)$ for any $s\in \mathbb{R}$. By means of the classical theory of symmetric hyperbolic quasi-linear systems (see for example \cite{MB}), the approximate system \eqref{APPBEBou} has a unique smooth solution $(u^{(n)},\theta^{(n)})$ on $[0,T]$ with some $T>0$, which may depend on $n$. Now we recall a blow-up criterion from \cite{Yena14}: if $\int_{0}^{T}\|\omega^{(n)}(t)\|_{L^{\infty}}\,dt<\infty$, the time $T$ can always be continued beyond. Then for every $n$, the \emph{a priori} estimate
\eqref{yyfvxaq226} guarantees that the solution $(u^{(n)},\theta^{(n)})$ is globally well-defined. Notice that for any $p\in[2,\infty]$
$$\|S_{n}\omega_{0}\|_{L^{p}}\leq \|\omega_{0}\|_{L^{p}},\quad \|S_{n}u_{0}\|_{L^{2}}\leq \|u_{0}\|_{L^{2}},$$
$$ \|S_{n}\theta_{0}\|_{H^{\rho}}\leq \|\theta_{0}\|_{H^{\rho}},\quad  \|S_{n}\mathcal{R}_{1}\theta_{0}\|_{L^{\infty}}\leq \|\mathcal{R}_{1}\theta_{0}\|_{L^{\infty}},$$
it is easy to see that the
\emph{a priori} estimates obtained in Propositions \ref{Gsds1}--\ref{Gsds3} for the approximate system \eqref{APPBEBou} are uniform in $n$, that is,
\begin{align}\label{addtgh01}
u^{(n)}\in L_{loc}^{\infty}(\mathbb{R}_{+}; H^{1}(\mathbb{R}^{2})),\quad \omega^{(n)} \in L_{loc}^{\infty}(\mathbb{R}_{+}; L^{\infty}(\mathbb{R}^{2})),\quad \theta^{(n)} \in L_{loc}^{\infty}(\mathbb{R}_{+}; \mathbb{L}(\mathbb{R}^{2})),
\end{align}
\begin{align}\label{addtgh02}
\mathcal {R}_{1}\theta^{(n)} \in L_{loc}^{\infty}(\mathbb{R}_{+}; L^{\infty}(\mathbb{R}^{2})),\quad \theta^{(n)}\in L_{loc}^{\infty}(\mathbb{R}_{+}; B_{r,1}^{\sigma}(\mathbb{R}^{2}))\cap L_{loc}^{1}(\mathbb{R}_{+}; B_{r,1}^{\sigma+1}(\mathbb{R}^{2})).
\end{align}
In order to show that $(u^{(n)},\theta^{(n)})_{n\in \mathbb{N}}$ converges (up to extraction), a boundedness information (the weak topology) over $(\partial_{t}u^{(n)},\partial_{t}\theta^{(n)})$ is required. To this end, we first claim that $(\partial_{t}u^{(n)})_{n\in \mathbb{N}}$ is bounded in $L_{loc}^{\infty}(\mathbb{R}_{+}; L^{2}(\mathbb{R}^{2}))$. Indeed, applying the Leray projector $\mathcal{P}$ over divergence free vector-fields to the velocity equation yields
$$\partial_{t}u^{(n)}=\mathcal{P}\left(\theta^{(n)} e_{2}-(u^{(n)} \cdot \nabla) u^{(n)}\right),$$
which together with the Calder$\rm\acute{o}$n--Zygmund theorem and interpolation inequality imply
\begin{align}
\|\partial_{t}u^{(n)}\|_{L^{2}}&=\|\mathcal{P}\left(\theta^{(n)} e_{2}-(u^{(n)} \cdot \nabla) u^{(n)}\right)\|_{L^{2}}\nonumber\\
&\leq
\|\theta^{(n)} e_{2}-(u^{(n)} \cdot \nabla) u^{(n)}\|_{L^{2}}\nonumber\\
&\leq\|\theta^{(n)}\|_{L^{2}}+\|(u^{(n)} \cdot \nabla) u^{(n)}\|_{L^{2}}\nonumber\\
&\leq\|\theta^{(n)}\|_{L^{2}}+\|u^{(n)}\|_{L^{p}}
\|\nabla u^{(n)}\|_{L^{\frac{2p}{p-2}}}
\nonumber\\
&\leq\|\theta^{(n)}\|_{L^{2}}+C\|u^{(n)}\|_{L^{p}}
\|\omega^{(n)}\|_{L^{\frac{2p}{p-2}}}
\nonumber\\
&\leq\|\theta^{(n)}\|_{L^{2}}+C
\|u^{(n)}\|_{L^{2}}^{\frac{2}{p}}
\|\omega^{(n)}\|_{L^{2}}^{1-\frac{2}{p}}
\|\omega^{(n)}\|_{L^{\frac{2p}{p-2}}}.\nonumber
\end{align}
Thanks to \eqref{addtgh01} and \eqref{addtgh02}, we thus obtain that $\partial_{t}u^{(n)}\in L_{loc}^{\infty}(\mathbb{R}_{+}; L^{2}(\mathbb{R}^{2}))$ uniform in $n$. As for the temperature, one may derive
\begin{align}
\|\partial_{t}\theta^{(n)}\|_{L_{t}^{2}H^{-1}}&\leq
 \|\Lambda\theta^{(n)}\|_{L_{t}^{2}H^{-1}}+
 \|(u^{(n)} \cdot \nabla)\theta^{(n)}\|_{L_{t}^{2}H^{-1}}+
 \|u_{2}^{(n)}\|_{L_{t}^{2}H^{-1}}
\nonumber\\
&\leq  \|\theta^{(n)}\|_{L_{t}^{2}L^{2}}+
 C\|u^{(n)} \theta^{(n)}\|_{L_{t}^{2}L^{2}}+
 \|u_{2}^{(n)}\|_{L_{t}^{2}L^{2}}
\nonumber\\
&\leq   \|\theta^{(n)}\|_{L_{t}^{2}L^{2}}+
 C\|u^{(n)}\|_{L_{t}^{\infty}L^{4}}\|\theta^{(n)}\|_{L_{t}^{2}L^{4}}+
 \|u_{2}^{(n)}\|_{L_{t}^{2}L^{2}}\nonumber\\
&\leq   \|\theta^{(n)}\|_{L_{t}^{2}L^{2}}+
 C\|u^{(n)}\|_{L_{t}^{\infty}L^{2}}^{\frac{1}{2}}
\|\omega^{(n)}\|_{L_{t}^{\infty}L^{2}}^{\frac{1}{2}}\|\theta^{(n)}\|_{L_{t}^{2}L^{4}}+
 \|u_{2}^{(n)}\|_{L_{t}^{2}L^{2}},\nonumber
\end{align}
which also along with \eqref{addtgh01} and \eqref{addtgh02} imply that $\partial_{t}\theta^{(n)}\in L_{loc}^{2}(\mathbb{R}_{+}; H^{-1}(\mathbb{R}^{2}))$ uniform in $n$. Because $H^{-1}(\mathbb{R}^{2})$ is (locally) compactly embedded in $L^{2}(\mathbb{R}^{2})$ the classical Aubin-Lions argument
(see e.g. \cite{RTemam}) ensures that, up to extraction, sequence $(\partial_{t}u^{(n)},\partial_{t}\theta^{(n)})_{n\in \mathbb{N}}$ strongly converges in $L_{loc}^{2}(\mathbb{R}_{+}; H_{loc}^{-1}(\mathbb{R}^{2}))$ to some function $u,\theta$ satisfying
$$u\in L_{loc}^{\infty}(\mathbb{R}_{+}; H^{1}(\mathbb{R}^{2})),\quad \omega \in L_{loc}^{\infty}(\mathbb{R}_{+}; L^{\infty}(\mathbb{R}^{2})),\quad \theta \in L_{loc}^{\infty}(\mathbb{R}_{+}; \mathbb{L}(\mathbb{R}^{2})),$$
$$\mathcal {R}_{1}\theta  \in L_{loc}^{\infty}(\mathbb{R}_{+}; L^{\infty}(\mathbb{R}^{2})),\quad \theta \in L_{loc}^{\infty}(\mathbb{R}_{+}; B_{r,1}^{\sigma}(\mathbb{R}^{2}))\cap L_{loc}^{1}(\mathbb{R}_{+}; B_{r,1}^{\sigma+1}(\mathbb{R}^{2})).$$
Interpolating with the uniform bounds \eqref{addtgh01} and \eqref{addtgh02}, it is clear to pass the limit in the approximate system \eqref{APPBEBou} and $(u,\theta)$ solves the system \eqref{BEBou} in the sense of distribution. Moreover, we can show $u\in \mathcal{C}(\mathbb{R}_{+}; H^{1}(\mathbb{R}^{2}))$ and $\theta\in \mathcal{C}(\mathbb{R}_{+}; B_{r,1}^{\sigma}(\mathbb{R}^{2}))$.
It suffices to consider $u\in C([0, T]; H^{1}(\mathbb{R}^{2}))$ for any $T>0$ as the same fashion can be applied to $\theta$ to obtain the desired result.
By the equivalent norm, we get
 \begin{align}\label{t2.03}
 \|u(t_{1})-u(t_{2})\|_{H^{1}}=&\left(\sum_{k<N}2^{2k}\|\Delta_{k}u(t_{1})
 -\Delta_{k}u(t_{2})\|_{L^{2}}^{2} \right)^{\frac{1}{2}}\nonumber\\&+ \left(\sum_{k\geq N}
 2^{2k}\|\Delta_{k}u(t_{1})-\Delta_{k}u(t_{2})\|_{L^{2}}^{2} \right)^{\frac{1}{2}}.
 \end{align}
Let $\epsilon>0$ be arbitrarily small. Thanks to
$$\sup_{0\leq t\leq T}\|u(t)\|_{H^{1}}\leq C(T)<\infty,$$
 there exists an integer
$N=N(\epsilon)>0$ such that
 \begin{align}\label{t2.04}
\left(\sum_{k\geq N}
 2^{2k}\|\Delta_{k}u(t_{1})-\Delta_{k}u(t_{2})\|_{L^{2}}^{2}
 \right)^{\frac{1}{2}}<\frac{\epsilon}{2}.
 \end{align}
According to the velocity equation, it implies
\begin{align}
\Delta_{k}u(t_{2})-\Delta_{k}u(t_{1}) =&\int_{t_{1}}^{t_{2}}{\frac{d}{d\tau}
\Delta_{k}u(\tau)\,d\tau}\nonumber\\
 =&-\int_{t_{1}}^{t_{2}}{ \Delta_{k}\mathcal {P}[(u\cdot\nabla) u- \theta](\tau)\,d\tau}.\nonumber
\end{align}
Consequently, we are able to show
\begin{align}
& \sum_{k<N}
 2^{2k} \|\Delta_{k}u(t_{1})-\Delta_{k}u(t_{2})\|_{L^{2}}^{2}
\nonumber\\
 &= \sum_{k<N}
 2^{2k}\Big\|\int_{t_{1}}^{t_{2}}{ \Delta_{k}\mathcal {P}[(u\cdot\nabla) u-\theta](\tau)\,d\tau}\Big\|_{L^{2}}^{2}\nonumber\\
&\leq \sum_{k<N}
 2^{2k}\Big(\int_{t_{1}}^{t_{2}}{ \|\Delta_{k}[(u\cdot\nabla) u-\theta]\|_{L^{2}}(\tau)\,d\tau}\Big)^{2}
\nonumber\\
&\leq \sum_{k<N}
 2^{2k}\Big(\int_{t_{1}}^{t_{2}}{ \big(\|\Delta_{k}(u\cdot\nabla u)(\tau)\|_{L^{2}}+
\|\Delta_{k}\theta(\tau)\|_{L^{2}}\big)\,d\tau}\Big)^{2} \nonumber\\
&\leq\sum_{k<N}
 2^{2k}\Big(\int_{t_{1}}^{t_{2}}{2^{k} \|\Delta_{k}(u\cdot\nabla u)(\tau)\|_{L^{1}}\,d\tau}\Big)^{2}
+
\sum_{k<N}2^{2k}
\Big(\int_{t_{1}}^{t_{2}}{
\|\Delta_{k}\theta(\tau)\|_{L^{2}}\,d\tau}\Big)^{2}
\nonumber\\
&\leq\sum_{k<N}
 2^{4k}\Big(\int_{t_{1}}^{t_{2}}{ \|(u\cdot\nabla u)(\tau)\|_{L^{1}}\,d\tau}\Big)^{2}
+
\sum_{k<N}2^{2k}
\Big(\int_{t_{1}}^{t_{2}}{
\|\Delta_{k}\theta(\tau)\|_{L^{2}}\,d\tau}\Big)^{2}
\nonumber\\
&\leq C\sum_{k<N}
 2^{4k} \|u\|_{L_{t}^{\infty}H^{1}}^{4}|t_{1}-t_{2}|^{2}  +C\sum_{k<N}
 2^{2 k} |t_{1}-t_{2}|^{2}
 \|\theta\|_{L_{t}^{\infty}L^{2}}^{2}  \nonumber\\
&\leq  C2^{4N} \|u\|_{L_{t}^{\infty}H^{1}}^{4}|t_{1}-t_{2}|^{2}  +C
 2^{2N} |t_{1}-t_{2}|^{2}
 \|\theta\|_{L_{t}^{\infty}L^{2}}^{2}\nonumber\\
&\leq  C(2^{4N} \|u\|_{L_{t}^{\infty}H^{1}}^{4} +
 2^{2N}
 \|\theta\|_{L_{t}^{\infty}L^{2}}^{2})|t_{1}-t_{2}|^{2},\nonumber
\end{align}
which implies
\begin{align}\label{cvnpuy9}
\sum_{k<N}
 2^{2k} \|\Delta_{k}u(t_{1})-\Delta_{k}u(t_{2})\|_{L^{2}}^{2}\leq  C\Big(2^{4N} \|u\|_{L_{t}^{\infty}H^{1}}^{4} +
 2^{2N}
 \|\theta\|_{L_{t}^{\infty}L^{2}}^{2}\Big)|t_{1}-t_{2}|^{2}.
 \end{align}
Taking $|t_{1}-t_{2}|$ small enough, we deduce from \eqref{cvnpuy9} that
\begin{align}\label{t2.07}
\left(\sum_{k< N}
 2^{2k}\|\Delta_{k}u(t_{1})-\Delta_{k}u(t_{2})\|_{L^{2}}^{2}
 \right)^{\frac{1}{2}}<\frac{\epsilon}{2}.
 \end{align}
Now the desired $u\in C([0, T]; H^{1}(\mathbb{R}^{2}))$ follows from (\ref{t2.03}), (\ref{t2.04}) and
(\ref{t2.07}). Hence, this completes the proof of existence.

\subsection{Uniqueness.}
The proof of the uniqueness part is more delicate. Due to the fact that we only have the bound of the vorticity in $L^{\infty}$-space, the proof of uniqueness is done by using Yudovich method (see \cite{Yudovich}) and the above derived estimates. More precisely, we let $(\widehat{u},\widehat{\theta},\widehat{p})$ and $(\widetilde{u},\widetilde{\theta},\widetilde{p})$ be two solutions of $(\ref{BEBou})$ enjoying the same initial data. We also denote
$$\delta u\triangleq\widehat{u}-\widetilde{u},\ \ \delta \theta\triangleq\widehat{\theta}-\widetilde{\theta},\ \ \delta p\triangleq\widehat{p}-\widetilde{p},$$
which satisfy
\begin{equation}\label{tdkp324}
\left\{\aligned
&\partial_{t}\delta u+(\widehat{u} \cdot \nabla)\delta u
+\nabla \delta p=\delta\theta e_{2}- ({\delta u} \cdot \nabla) \widetilde{u}, \\
&\partial_{t}\delta\theta+(\widehat{u} \cdot \nabla) \delta\theta+\Lambda \delta\theta=\delta u_{2}- ({\delta u} \cdot \nabla) \widetilde{\theta},  \\
&\nabla\cdot \delta u=0, \\
&\delta u(x,0)=0,\quad \delta \theta(x,0)=0.
\endaligned\right.
\end{equation}
As the assumption on initial temperature $\theta_{0}$ only satisfies $\mathcal {R}_{1}\theta_{0}\in L^{\infty}(\mathbb{R}^{2})$ and $\theta_{0}\in  \mathbb{L}(\mathbb{R}^{2})$, without any regularity assumption on it, proving uniqueness is more involved and it requires some new idea. More precisely, the direct $L^{2}$-energy method is not workable.
To overcome this difficulty, we introduce the auxiliary quantities $\widehat{\eta}$ and $\widetilde{\eta}$ given by
$$\Lambda \widehat{\eta}=\widehat{\theta},\quad \Lambda \widetilde{\eta}=\widetilde{\theta},\quad \delta\eta=\widehat{\eta}-\widetilde{\eta},$$
which yields
$$\Lambda \delta\eta=\delta\theta.$$
We derive from $\eqref{tdkp324}_{2}$ that
\begin{align} \label{di23710}
\partial_{t}\Lambda \delta\eta+(\widehat{u} \cdot \nabla) \Lambda \delta\eta+\Lambda^{2} \delta\eta= \delta u_{2}- ({\delta u} \cdot \nabla) \widetilde{\theta}.
\end{align}
Dotting \eqref{di23710} by $\delta\eta$, we deduce
\begin{align} \label{di23711}
\frac{1}{2}\frac{d}{dt}\|\Lambda^{\frac{1}{2}}\delta\eta(t)\|_{L^{2}}^{2}
+\|\Lambda \delta\eta \|_{L^{2}}^{2}=&\int_{\mathbb{R}^{2}}{\delta u_{2}
\,\delta\eta\,dx}-
\int_{\mathbb{R}^{2}}{(\widehat{u} \cdot \nabla) \Lambda
\delta\eta\,\delta\eta\,dx}\nonumber\\&- \int_{\mathbb{R}^{2}}{({\delta u} \cdot
\nabla) \widetilde{\theta}\,\delta\eta\,dx}.
\end{align}
From a standard $L^2$-energy method, we deduce from $\eqref{tdkp324}_{1}$ that
 \begin{align} \label{di23701}
\frac{1}{2}\frac{d}{dt}\|\delta u(t)\|_{L^{2}}^{2}=-
\int_{\mathbb{R}^{2}}{ ({\delta u} \cdot \nabla) \widetilde{u}\cdot
\delta u\,dx} +\int_{\mathbb{R}^{2}}{ \delta\theta  \delta
u_{2}\,dx}.
\end{align}
Summing up \eqref{di23711} and \eqref{di23701} shows
\begin{align} \label{drg239}
&\frac{1}{2}\frac{d}{dt}(\|\Lambda^{\frac{1}{2}}\delta\eta(t)\|_{L^{2}}^{2}+\|\delta u(t)\|_{L^{2}}^{2})
+\|\Lambda \delta\eta \|_{L^{2}}^{2}\nonumber\\&=\int_{\mathbb{R}^{2}}{\delta u_{2}
\,\delta\eta\,dx}-
\int_{\mathbb{R}^{2}}{(\widehat{u} \cdot \nabla) \Lambda
\delta\eta\,\delta\eta\,dx}- \int_{\mathbb{R}^{2}}{({\delta u} \cdot
\nabla) \widetilde{\theta}\,\delta\eta\,dx}\nonumber\\&\quad-
\int_{\mathbb{R}^{2}}{ ({\delta u} \cdot \nabla) \widetilde{u}\cdot
\delta u\,dx}+\int_{\mathbb{R}^{2}}{ \delta\theta  \delta
u_{2}\,dx}\nonumber\\&\triangleq \sum_{m=1}^{5}A_{m}.
\end{align}
By virtue of $\nabla\cdot \widehat{u}=0$, it yields
\begin{align}
A_{2}&=-
\int_{\mathbb{R}^{2}}{ \widehat{u}_{i} \partial_{i}\Lambda
\delta\eta\,\delta\eta\,dx}
\nonumber\\&
=
\int_{\mathbb{R}^{2}}{ \widehat{u}_{i} \partial_{i}\delta\eta \,\Lambda
\delta\eta\,dx}
\nonumber\\&
=\int_{\mathbb{R}^{2}}{(\widehat{u}
\cdot \nabla)  \delta\eta\,\Lambda\delta\eta\,dx}\nonumber\\&=
\int_{\mathbb{R}^{2}}{\Lambda^{\frac{1}{2}}\{(\widehat{u} \cdot
\nabla)  \delta\eta\}\,\Lambda^{\frac{1}{2}}\delta\eta\,dx}
\nonumber\\&
=\int_{\mathbb{R}^{2}}{[\Lambda^{\frac{1}{2}},\widehat{u}\cdot
\nabla]  \delta\eta \,\Lambda^{\frac{1}{2}}\delta\eta\,dx}, \nonumber
\end{align}
which along with \eqref{cmoi78ye1} yields
\begin{align}
\left|A_{2}\right| &
=\left|\int_{\mathbb{R}^{2}}{[\Lambda^{\frac{1}{2}},\widehat{u}\cdot
\nabla]  \delta\eta \,\Lambda^{\frac{1}{2}}\delta\eta\,dx}\right|
\nonumber\\& \leq
C\|[\Lambda^{\frac{1}{2}},\widehat{u}\cdot
\nabla]  \delta\eta\|_{L^{\frac{3}{2}}}\|\Lambda^{\frac{1}{2}}\delta \eta\|_{L^{3}}
\nonumber\\& \leq   C(\|\nabla
\widehat{u}\|_{L^{6}}\|\Lambda^{\frac{1}{2}}\delta
\eta\|_{L^{2}}+\|\nabla\delta \eta\|_{L^{2}}\|\Lambda^{\frac{1}{2}}
\widehat{u}\|_{L^{6}})\|\Lambda^{\frac{1}{2}}\delta \eta\|_{L^{2}}^{\frac{1}{3}}\|\Lambda \delta \eta\|_{L^{2}}^{\frac{2}{3}}
\nonumber\\& \leq   C(\|\widehat{\omega}\|_{L^{6}}\|\Lambda^{\frac{1}{2}}\delta
\eta\|_{L^{2}}+\|\Lambda \delta
\eta\|_{L^{2}}\|\widehat{\omega}\|_{L^{\frac{12}{5}}})\|\Lambda^{\frac{1}{2}}\delta \eta\|_{L^{2}}^{\frac{1}{3}}\|\Lambda \delta \eta\|_{L^{2}}^{\frac{2}{3}}
\nonumber\\
&\leq \frac{1}{8}\|\Lambda \delta\eta \|_{L^{2}}^{2}+C(\|\widehat{\omega}\|_{L^{6}}^{\frac{3}{2}}
+\|\widehat{\omega}\|_{L^{\frac{12}{5}}}^{6})
\|\Lambda^{\frac{1}{2}}\delta
\eta\|_{L^{2}}^{2}.\nonumber
\end{align}
Invoking \eqref{LOG2} and \eqref{yydkl028}, we have
\begin{align}
\left|A_{3}\right| &
=\left|\int_{\mathbb{R}^{2}}{{\delta u}_{i}\partial_{i}\widetilde{\theta}\,\delta\eta\,dx}\right|
\nonumber\\& =
\left|\int_{\mathbb{R}^{2}}{{\delta u}_{i}\widetilde{\theta}\,\partial_{i}\delta\eta\,dx}\right|
\nonumber\\& \leq
C\|\widetilde{\theta}\|_{L^{\infty}}\|\delta u\|_{L^{2}}\|\Lambda\delta \eta\|_{L^{2}}\nonumber\\
&\leq \frac{1}{8}\|\Lambda \delta\eta \|_{L^{2}}^{2}+C\|\widetilde{\theta}\|_{L^{\infty}}^{2}
\|\delta u\|_{L^{2}}^{2}
\nonumber\\
&\leq \frac{1}{8}\|\Lambda \delta\eta \|_{L^{2}}^{2}+C\left(1+\sup_{m\geq2}\frac{\|\widetilde{\theta}\|_{L^{m}}}{m}
\ln\big(e+\|\widetilde{\theta}\|_{B_{r,1}^{\sigma+1}}\big)\right)^{2}
\|\delta u\|_{L^{2}}^{2}
\nonumber\\
&\leq \frac{1}{8}\|\Lambda \delta\eta \|_{L^{2}}^{2}+C\left(1+\left(\sup_{m\geq2}\frac{\|\widetilde{\theta}\|_{L^{m}}}{m}\right)^{2}
\left(\ln\big(e+\|\widetilde{\theta}\|_{B_{r,1}^{\sigma+1}}\big)\right)^{2}\right)
\|\delta u\|_{L^{2}}^{2}
\nonumber\\
&\leq \frac{1}{8}\|\Lambda \delta\eta \|_{L^{2}}^{2}+C\left(1+\left(\sup_{m\geq2}\frac{\|\widetilde{\theta}\|_{L^{m}}}{m}\right)^{2}
\left(\ln\big(e+\|\widetilde{\theta}\|_{B_{r,1}^{\sigma+1}}\big)\right)^{2}\right)
\|\delta u\|_{L^{2}}^{2}
\nonumber\\
&\leq \frac{1}{8}\|\Lambda \delta\eta \|_{L^{2}}^{2}+C\left(1+\left(\sup_{m\geq2}\frac{\|\widetilde{\theta}\|_{L^{m}}}{m}\right)^{2}\right)
(1+\|\widetilde{\theta}\|_{B_{r,1}^{\sigma+1}})
\|\delta u\|_{L^{2}}^{2}\nonumber\\
&\leq \frac{1}{8}\|\Lambda \delta\eta \|_{L^{2}}^{2}+C
(1+\|\widetilde{\theta}\|_{B_{r,1}^{\sigma+1}})
\|\delta u\|_{L^{2}}^{2},\nonumber
\end{align}
where we have used the simple fact
$$ \ln\big(e+\|\widetilde{\theta}\|_{B_{r,1}^{\sigma+1}}\big) \leq C\Big(1+\|\widetilde{\theta}\|_{B_{r,1}^{\sigma+1}}\Big)^{\frac{1}{2}}.$$
For the term $A_{4}$, we conclude
\begin{align}
\left|A_{4}\right|
\leq& C\|\nabla \widetilde{u}\|_{L^{r}} \|\delta u\|_{L^{\frac{2r}{r-1}}}^{2}
\nonumber\\
 \leq&Cr\|\widetilde{\omega}\|_{L^{r}} \|\delta u\|_{L^{\frac{2r}{r-1}}}^{2}
 \nonumber\\
 \leq&Cr\|\widetilde{\omega}\|_{L^{r}} \|\delta u\|_{L^{2}}^{\frac{2r-2}{r}}
 \|\delta u\|_{L^{\infty}}^{\frac{2}{r}} \nonumber\\
 \leq&Cr(\|\widetilde{\omega}\|_{L^{2}}+\|\widetilde{\omega}\|_{L^{\infty}}) \|\delta u\|_{L^{2}}^{\frac{2r-2}{r}}
 \|\delta u\|_{L^{\infty}}^{\frac{2}{r}}.\nonumber
\end{align}
In view of the basic energy estimate and \eqref{yyfvxaq226}, we have that
\begin{align}
\|\delta u(t)\|_{L^{\infty}}
  \leq C(t,\,u_{0},\,\theta_{0}).\nonumber
\end{align}
The term $A_{5}$ can be bounded by
\begin{align}
\left|A_{5}\right| &
=\left|\int_{\mathbb{R}^{2}}{ \Lambda \delta\eta \delta
u_{2}\,dx}\right|
\nonumber\\& \leq
C\|\delta u\|_{L^{2}}\|\Lambda\delta \eta\|_{L^{2}}\nonumber\\
&\leq \frac{1}{8}\|\Lambda \delta\eta \|_{L^{2}}^{2}+C
\|\delta u\|_{L^{2}}^{2}.\nonumber
\end{align}
Putting the above estimates into \eqref{drg239} yields
\begin{align} \label{vkdxza11}
&\frac{d}{dt}(\|\Lambda^{\frac{1}{2}}\delta\eta(t)\|_{L^{2}}^{2}+\|\delta u(t)\|_{L^{2}}^{2})
+\|\Lambda \delta\eta \|_{L^{2}}^{2}\nonumber\\&\leq CA_{1}+Cr(\|\widetilde{\omega}\|_{L^{2}}+\|\widetilde{\omega}\|_{L^{\infty}}) \|\delta u\|_{L^{2}}^{\frac{2r-2}{r}}
 \|\delta u\|_{L^{\infty}}^{\frac{2}{r}}\nonumber\\&\quad+C\left(1+\|\widehat{\omega}\|_{L^{6}}^{\frac{3}{2}}
+\|\widehat{\omega}\|_{L^{\frac{12}{5}}}^{6}+\|\widetilde{\theta}\|_{L^{2}}^{2}+
\|\widetilde{\theta}\|_{B_{r,1}^{\sigma+1}}\right)
(\|\Lambda^{\frac{1}{2}}\delta\eta\|_{L^{2}}^{2}+\|\delta u\|_{L^{2}}^{2}).
\end{align}
Unfortunately, the term $A_{1}$ can not be handled as above. To close \eqref{vkdxza11}, we need to estimate $\|\Lambda^{-\frac{1}{2}}\delta u\|_{L^{2}}$. To this end, we dot $\eqref{tdkp324}_{1}$ by $\Lambda^{-1}\delta u$ to conclude
\begin{align}
 \frac{d}{dt}\|\Lambda^{-\frac{1}{2}}\delta u(t)\|_{L^{2}}^{2}=&-2
\int_{\mathbb{R}^{2}}{ (\widehat{u} \cdot \nabla)\delta u\cdot\Lambda^{-1}
\delta u\,dx}-2
\int_{\mathbb{R}^{2}}{ ({\delta u} \cdot \nabla) \widetilde{u}\cdot\Lambda^{-1}
\delta u\,dx} \nonumber\\&+2\int_{\mathbb{R}^{2}}{ \delta\theta  \Lambda^{-1}\delta
u_{2}\,dx}\nonumber\\=&-2
\int_{\mathbb{R}^{2}}{ \widehat{u}_{i} \partial_{i}\delta u_{j} \Lambda^{-1}
\delta u_{j}\,dx}-2
\int_{\mathbb{R}^{2}}{ {\delta u}_{i} \partial_{i}  \widetilde{u}_{j} \Lambda^{-1}
\delta u_{j}\,dx} \nonumber\\&+2\int_{\mathbb{R}^{2}}{ \Lambda \delta\eta  \Lambda^{-1}\delta
u_{2}\,dx}\nonumber\\=&2
\int_{\mathbb{R}^{2}}{ \widehat{u}_{i}\delta u_{j} \partial_{i}\Lambda^{-1}
\delta u_{j}\,dx}+2
\int_{\mathbb{R}^{2}}{ {\delta u}_{i} \widetilde{u}_{j} \partial_{i} \Lambda^{-1}
\delta u_{j}\,dx} \nonumber\\&+2\int_{\mathbb{R}^{2}}{ \Lambda \delta\eta  \Lambda^{-1}\delta
u_{2}\,dx}\nonumber\\& \leq
C(\|\widehat{u}\|_{L^{\infty}}+\|\widetilde{u}\|_{L^{\infty}})\|\delta u\|_{L^{2}}\|\partial_{i}\Lambda^{-1}
\delta u\|_{L^{2}}+C\|\Lambda^{\frac{1}{2}}\delta\eta\|_{L^{2}}
\|\Lambda^{-\frac{1}{2}}\delta u\|_{L^{2}}
\nonumber\\& \leq
C(1+\|\widehat{u}\|_{L^{\infty}}+\|\widetilde{u}\|_{L^{\infty}})(\|\delta u\|_{L^{2}}^{2}+\|\Lambda^{\frac{1}{2}}\delta\eta\|_{L^{2}}^{2}+
\|\Lambda^{-\frac{1}{2}}\delta u\|_{L^{2}}^{2}),\nonumber
\end{align}
which along with \eqref{vkdxza11} implies
\begin{align} \label{vkdxza12}
&\frac{d}{dt}(\|\Lambda^{\frac{1}{2}}\delta\eta(t)\|_{L^{2}}^{2}+\|\delta u(t)\|_{L^{2}}^{2}+\|\Lambda^{-\frac{1}{2}}\delta u(t)\|_{L^{2}}^{2})
+\|\Lambda \delta\eta \|_{L^{2}}^{2}\nonumber\\&\leq CA_{1}+Cr(\|\widetilde{\omega}\|_{L^{2}}+\|\widetilde{\omega}\|_{L^{\infty}}) \|\delta u\|_{L^{2}}^{\frac{2r-2}{r}}
 \|\delta u\|_{L^{\infty}}^{\frac{2}{r}}\nonumber\\&\quad
 +C(1+\|\widehat{\omega}\|_{L^{6}}^{\frac{3}{2}}
+\|\widehat{\omega}\|_{L^{\frac{12}{5}}}^{6}+\|\widetilde{\theta}\|_{L^{2}}^{2}
+\|\widetilde{\theta}\|_{B_{r,1}^{\sigma+1}}
+\|\widehat{u}\|_{L^{\infty}}+\|\widetilde{u}\|_{L^{\infty}})\nonumber\\ &\quad \quad\times
(\|\Lambda^{\frac{1}{2}}\delta\eta\|_{L^{2}}^{2}+\|\delta u\|_{L^{2}}^{2}+
\|\Lambda^{-\frac{1}{2}}\delta u\|_{L^{2}}^{2}).
\end{align}
Let us now move to the term $A_{1}$, which can be easily bounded by
\begin{align}\label{vkdxza13}
\left|A_{1}\right| &
=\left|\int_{\mathbb{R}^{2}}{\Lambda^{-\frac{1}{2}}\delta u_{2}
\,\Lambda^{\frac{1}{2}}\delta\eta\,dx}\right|
\nonumber\\& \leq
C\|\Lambda^{-\frac{1}{2}}\delta u\|_{L^{2}}\|\Lambda^{\frac{1}{2}}\delta \eta\|_{L^{2}}\nonumber\\
&\leq C(\|\Lambda^{-\frac{1}{2}}\delta u\|_{L^{2}}^{2}+
\|\Lambda^{\frac{1}{2}}\delta \eta\|_{L^{2}}^{2}).
\end{align}
Inserting \eqref{vkdxza13} into \eqref{vkdxza12}, we finally get
\begin{align} \label{vkdxza14}
&\frac{d}{dt}(\|\Lambda^{\frac{1}{2}}\delta\eta(t)\|_{L^{2}}^{2}+\|\delta u(t)\|_{L^{2}}^{2}+\|\Lambda^{-\frac{1}{2}}\delta u(t)\|_{L^{2}}^{2})
\nonumber\\&\leq Cr(\|\widetilde{\omega}\|_{L^{2}}+\|\widetilde{\omega}\|_{L^{\infty}}) \|\delta u\|_{L^{2}}^{\frac{2r-2}{r}}
 \|\delta u\|_{L^{\infty}}^{\frac{2}{r}}\nonumber\\&\quad
 +L(t)
(\|\Lambda^{\frac{1}{2}}\delta\eta\|_{L^{2}}^{2}+\|\delta u\|_{L^{2}}^{2}+
\|\Lambda^{-\frac{1}{2}}\delta u\|_{L^{2}}^{2}),
\end{align}
where $L(t)$ is given by
$$L(t)\triangleq C\left(1+\|\widehat{\omega}\|_{L^{6}}^{\frac{3}{2}}
+\|\widehat{\omega}\|_{L^{\frac{12}{5}}}^{6}+\|\widetilde{\theta}\|_{L^{2}}^{2}
+\|\widetilde{\theta}\|_{B_{r,1}^{\sigma+1}}
+\|\widehat{u}\|_{L^{\infty}}+\|\widetilde{u}\|_{L^{\infty}}\right)(t).$$
To simplify our presentation, we denote
$$ X_{\epsilon}(t)\triangleq \|\Lambda^{\frac{1}{2}}\delta\eta(t)\|_{L^{2}}^{2}+\|\delta u(t)\|_{L^{2}}^{2}+\|\Lambda^{-\frac{1}{2}}\delta u(t)\|_{L^{2}}^{2}+\epsilon,$$
where $\epsilon>0$ is a small parameter.
It thus follows from \eqref{vkdxza14} that
\begin{align}
  \frac{d}{dt}X_{\epsilon}(t) \leq L(t)X_{\epsilon}(t)+ Cr(\|\widetilde{\omega}\|_{L^{2}}+\|\widetilde{\omega}\|_{L^{\infty}})
 \|\delta u\|_{L^{\infty}}^{\frac{2}{r}}X_{\epsilon}(t)^{\frac{r-1}{r}},\nonumber
\end{align}
which yields
\begin{align} \label{tdkp327}
  \frac{d}{dt}\left(e^{-\int_{0}^{t}L(\tau)\,d\tau}X_{\epsilon}(t)\right) \leq & Cr(\|\widetilde{\omega}\|_{L^{2}}+\|\widetilde{\omega}\|_{L^{\infty}})
 \|\delta u\|_{L^{\infty}}^{\frac{2}{r}}
 X_{\epsilon}(t)^{\frac{r-1}{r}}e^{-\int_{0}^{t}L(\tau)\,d\tau}.
\end{align}
Setting
$$Y_{\epsilon}(t)\triangleq e^{-\int_{0}^{t}L(\tau)\,d\tau}X_{\epsilon}(t),$$
we thus get from \eqref{tdkp327} that
\begin{align}
  \frac{d}{dt}Y_{\epsilon}(t) \leq  Cr(\|\widetilde{\omega}\|_{L^{2}}+\|\widetilde{\omega}\|_{L^{\infty}})
 \|\delta u\|_{L^{\infty}}^{\frac{2}{r}}
 Y_{\epsilon}(t)^{\frac{r-1}{r}}e^{-\frac{1}{r}\int_{0}^{t}L(\tau)\,d\tau}\nonumber
\end{align}
or
\begin{align}
  Y_{\epsilon}(t)^{\frac{1}{r}-1}\frac{d}{dt}Y_{\epsilon}(t) \leq  Cr(\|\widetilde{\omega}\|_{L^{2}}+\|\widetilde{\omega}\|_{L^{\infty}})
 \|\delta u\|_{L^{\infty}}^{\frac{2}{r}}
 e^{-\frac{1}{r}\int_{0}^{t}L(\tau)\,d\tau}.\nonumber
\end{align}
Integrating it in time implies
\begin{align}
  Y_{\epsilon}(t)\leq  \left(\epsilon^{\frac{1}{r}}+ C\int_{0}^{t}(\|\widetilde{\omega}(\varrho)\|_{L^{2}}+\|\widetilde{\omega}(\varrho)
  \|_{L^{\infty}})
 \|\delta u(\varrho)\|_{L^{\infty}}^{\frac{2}{r}}
 e^{-\frac{1}{r}\int_{0}^{\varrho}L(\tau)\,d\tau}\,d\varrho\right)^{r}.\nonumber
\end{align}
Letting $\epsilon\rightarrow0$, we achieve
\begin{align}  \label{tdkp328}
&\|\Lambda^{\frac{1}{2}}\delta\eta(t)\|_{L^{2}}^{2}+\|\delta u(t)\|_{L^{2}}^{2}+\|\Lambda^{-\frac{1}{2}}\delta u(t)\|_{L^{2}}^{2}\nonumber\\ &\leq e^{\int_{0}^{t}L(\tau)\,d\tau}\left( C\int_{0}^{t}(\|\widetilde{\omega}(\varrho)\|_{L^{2}}+\|\widetilde{\omega}(\varrho)
  \|_{L^{\infty}})
 \|\delta u(\varrho)\|_{L^{\infty}}^{\frac{2}{r}}
 \,d\varrho\right)^{r}\nonumber\\
&\leq e^{\int_{0}^{t}L(\tau)\,d\tau} \|\delta u\|_{L_{t}^{\infty}L^{\infty}}^{2} \left( C\int_{0}^{t}(\|\widetilde{\omega}(\varrho)\|_{L^{2}}+\|\widetilde{\omega}(\varrho)
  \|_{L^{\infty}})
 \,d\varrho\right)^{r}.
 \end{align}
Bearing in mind all the above obtained estimates, it is not hard to check that $L(t)$ belongs to $L_{loc}^{1}(\mathbb{R}_{+})$.
Due to \eqref{tt5add01} and \eqref{yyfvxaq226}, we are able to show that there exists a positive time $\widetilde{T}$ such that $$C\int_{0}^{\widetilde{T}}(\|\widetilde{\omega}(\varrho)\|_{L^{2}}+\|\widetilde{\omega}(\varrho)
  \|_{L^{\infty}})
 \,d\varrho<1.$$
Letting $r\rightarrow\infty$ in \eqref{tdkp328} entails
that $\delta u=0,\,\Lambda^{\frac{1}{2}}\delta\eta=0$ on $[0,\widetilde{T}]$. Notice that $\delta u$ and $\Lambda^{\frac{1}{2}}\delta\eta$ are continuous in time with values in $L^{2}$, by means of a standard connectivity argument, we may conclude that $\delta u=0,\,\Lambda^{\frac{1}{2}}\delta\eta=0$ on $\mathbb{R}_{+}$. As $\delta\theta=\Lambda \delta\eta=\Lambda^{\frac{1}{2}}(\Lambda^{\frac{1}{2}}\delta\eta)$, we thus have $\delta\theta=0$ on $\mathbb{R}_{+}$.
Consequently, we complete the proof of Theorem \ref{Th1}.

\vskip .3in
\appendix
\section{ The system (\ref{BEBou}) with regular initial data} \label{appSec2}
In this appendix, we consider the case of the system (\ref{BEBou}) with regular initial data. More precisely, we have the following result.
\begin{thm}\label{addTh4}
Let $u_{0}$ be a divergence-free vector field and $u_{0}\in H^{s}(\mathbb{R}^{2})$, $\theta_{0}\in H^{\widetilde{s}}(\mathbb{R}^{2})$ with $|s-\widetilde{s}|\leq\frac{1}{2}$ and $s>2$. Then (\ref{BEBou}) has a unique global solution $(u,\theta)$ such that
$$u\in \mathcal{C}(\mathbb{R}_{+}; H^{s}(\mathbb{R}^{2})),\quad \theta\in \mathcal{C}(\mathbb{R}_{+}; H^{\widetilde{s}}(\mathbb{R}^{2})) \cap L_{loc}^{2}(\mathbb{R}_{+}; H^{\widetilde{s}+\frac{1}{2}}(\mathbb{R}^{2})).$$
\end{thm}

\begin{proof}
As the local well-posedness can be established via the standard approach, it suffices to show the required \emph{a priori} estimates. Moreover, due to the higher regular initial data, we choose not to apply the Littlewood-Paley decomposition to the $\theta$-equation itself but to use the interpolation inequalities and
the classical commutator estimates. Now let us sketch the proof.
To begin with, it follows from \eqref{dfpyt78} that
\begin{align}\label{etyp001}
 \frac{d}{dt}\|\nabla\theta(t)\|_{L^{2}}^{2}
+\|\Lambda^{\frac{3}{2}}\theta\|_{L^{2}}^{2}
&\leq C(1+\|\Lambda^{\frac{1}{2}}\theta\|_{L^{2}}^{2})(
\|\omega\|_{L^{4}}^{4}+\|\omega\|_{L^{2}}^{2}+
\|\nabla \theta\|_{L^{2}}^{2})\nonumber\\&\leq C(1+\|\Lambda^{\frac{1}{2}}\theta\|_{L^{2}}^{2})(
\|\Gamma\|_{L^{4}}^{4}+\|\theta\|_{L^{4}}^{4}+\|\Gamma\|_{L^{2}}^{2}+\|\theta\|_{L^{2}}^{2}+
\|\nabla \theta\|_{L^{2}}^{2})\nonumber\\&\leq C(1+\|\Lambda^{\frac{1}{2}}\theta\|_{L^{2}}^{2})(
\|\Gamma\|_{L^{4}}^{4}+\|\Gamma\|_{L^{2}}^{2}+
\|\nabla \theta\|_{L^{2}}^{2}).
\end{align}
Recalling \eqref{fzhpet17} and \eqref{vbjlwq6}, it is not hard to show
\begin{align}\label{etyp002}
 \frac{d}{dt}(\|\Gamma(t)\|_{L^{4}}^{4}+\|\Gamma(t)\|_{L^{2}}^{2})
\leq C(1+\|\theta\|_{L^{\infty}})(1+\|\Gamma\|_{L^{4}}^{4}+\|\Gamma\|_{L^{2}}^{2}).
\end{align}
Adding \eqref{etyp001} and \eqref{etyp002} together implies
\begin{align}\label{etyp003}
 &\frac{d}{dt}(\|\nabla\theta(t)\|_{L^{2}}^{2}+\|\Gamma(t)\|_{L^{4}}^{4}+\|\Gamma(t)\|_{L^{2}}^{2})
 +\|\Lambda^{\frac{3}{2}}\theta\|_{L^{2}}^{2}
\nonumber\\& \leq C(1+\|\Lambda^{\frac{1}{2}}\theta\|_{L^{2}}^{2}+\|\theta\|_{L^{\infty}})
(1+\|\Gamma\|_{L^{4}}^{4}+\|\Gamma\|_{L^{2}}^{2}+\|\nabla \theta\|_{L^{2}}^{2}).
\end{align}
Following the same lines in proving Proposition \ref{Pro5add1}, we deduce from \eqref{etyp003} that
$$\|\nabla\theta(t)\|_{L^{2}}^{2}+\|\Gamma(t)\|_{L^{4}}^{4}+\|\Gamma(t)\|_{L^{2}}^{2}
+\int_{0}^{t}{\|\Lambda^{\frac{3}{2}}\theta(\tau)\|_{L^{2}}^{2}\,d\tau}\leq C(t,\,u_{0},\,\theta_{0}),$$
which yields
\begin{align}\label{etyp004}
\|\nabla\theta(t)\|_{L^{2}}^{2}+\|\omega(t)\|_{L^{4}}^{4}+\|\omega(t)\|_{L^{2}}^{2}
+\int_{0}^{t}{\|\Lambda^{\frac{3}{2}}\theta(\tau)\|_{L^{2}}^{2}\,d\tau}\leq C(t,\,u_{0},\,\theta_{0}).
\end{align}
Keeping in mind \eqref{etyp004}, one derives from Proposition \ref{Gsds3} that
\begin{align}\label{etyp005}
\|\omega(t)\|_{L^{\infty}}
\leq C(t,\,u_{0},\,\theta_{0}).
\end{align}
It thus follows from \eqref{etyp004} and \eqref{etyp005} that
\begin{align}\label{tdkp329}
\|\omega(t)\|_{L^{2}}+\|\omega(t)\|_{L^{\infty}}
\leq C(t,\,u_{0},\,\theta_{0}).
\end{align}
Due to $|s-\widetilde{s}|\leq\frac{1}{2}$ and $s>2$, one verifies that $\widetilde{s}>\frac{3}{2}$. Therefore, we next claim
\begin{eqnarray}\label{yrtyp03}
\|\Lambda^{\frac{3}{2}}\theta(t)\|_{L^{2}}^{2}+ \int_{0}^{t}{\|\Delta\theta(\tau)\|_{L^{2}}^{2}\,d\tau}\leq C(t,\,u_{0},\,\theta_{0}).
\end{eqnarray}
To prove \eqref{yrtyp03}, we get by applying $\Lambda^{\frac{3}{2}}$ to $(\ref{BEBou})_{2}$ and multiplying it by $\Lambda^{\frac{3}{2}}\theta$ that
\begin{align}\label{yrtyp01}
\frac{1}{2}\frac{d}{dt}\|\Lambda^{\frac{3}{2}}\theta(t)\|_{L^{2}}^{2}
+\|\Delta\theta\|_{L^{2}}^{2}=-\int_{\mathbb{R}^{2}}
\Lambda^{\frac{3}{2}}\big(u \cdot
\nabla\theta\big)\Lambda^{\frac{3}{2}}\theta\,dx+\int_{\mathbb{R}^{2}}
\Lambda^{\frac{3}{2}}u_{2}\Lambda^{\frac{3}{2}}\theta\,dx.
\end{align}
Obviously, one has
\begin{align}\label{yghp854}
\left|\int_{\mathbb{R}^{2}}
\Lambda^{\frac{3}{2}}u_{2}\Lambda^{\frac{3}{2}}\theta\,dx\right|&\leq C\|\Lambda u\|_{L^{2}}\|\Delta\theta\|_{L^{2}}\nonumber\\&\leq\frac{1}{4}
\|\Delta\theta\|_{L^{2}}^{2}+C\|\omega\|_{L^{2}}^{2}.
\end{align}
By easy computations, we get
\begin{align}
\int_{\mathbb{R}^{2}}
\Lambda^{\frac{3}{2}}\big(u \cdot
\nabla\theta\big)\Lambda^{\frac{3}{2}}\theta\,dx&=\int_{\mathbb{R}^{2}}
\Lambda^{\frac{3}{2}}(u_{i} \partial_{i}\theta)\Lambda^{\frac{3}{2}}\theta\,dx
\nonumber\\&=\int_{\mathbb{R}^{2}}
\Big(\Lambda^{\frac{3}{2}}(u_{i} \partial_{i}\theta)-\Lambda^{\frac{3}{2}}u_{i} \partial_{i}\theta-u_{i}\Lambda^{\frac{3}{2}} \partial_{i}\theta\Big)\Lambda^{\frac{3}{2}}\theta\,dx\nonumber\\&\quad
+\int_{\mathbb{R}^{2}}\Lambda^{\frac{3}{2}}u_{i} \partial_{i}\theta\Lambda^{\frac{3}{2}}\theta\,dx+
\int_{\mathbb{R}^{2}}u_{i}\Lambda^{\frac{3}{2}} \partial_{i}\theta\Lambda^{\frac{3}{2}}\theta\,dx
\nonumber\\&=\int_{\mathbb{R}^{2}}
\Big(\Lambda^{\frac{3}{2}}(u_{i} \partial_{i}\theta)-\Lambda^{\frac{3}{2}}u_{i} \partial_{i}\theta-u_{i}\Lambda^{\frac{3}{2}} \partial_{i}\theta\Big)\Lambda^{\frac{3}{2}}\theta\,dx\nonumber\\&\quad
+\int_{\mathbb{R}^{2}}\Lambda^{\frac{3}{2}}u_{i} \partial_{i}\theta\Lambda^{\frac{3}{2}}\theta\,dx
\nonumber\\&\triangleq \Gamma_{1}+\Gamma_{2}.\nonumber
\end{align}
To bound $\Gamma_{1}$, we recall the following commutator estimate (see \cite[Corollary 1.1]{Fujiwara})
\begin{eqnarray}\label{tdfhklsd2}
\|\Lambda^{s}(fg)-g\Lambda^{s}f-f\Lambda^{s}g\|_{L^{p}}\leq C\|\Lambda^{s_{1}} f\|_{L^{r_{1}}}\|\Lambda^{s_{2}}g\|_{L^{r_{2}}},
\end{eqnarray}
where $s=s_{1}+s_{2}$ with $0\leq s_{1},\,s_{2}\leq1$ and $p, r_{1}, r_{2}\in(1, \infty)$ such that $\frac{1}{p}=\frac{1}{r_{1}}+\frac{1}{r_{2}}$.
In view of \eqref{tdfhklsd2}, we may derive
\begin{align}
|\Gamma_{1}|&\leq C\|\Lambda^{\frac{3}{2}}(u_{i} \partial_{i}\theta)-\Lambda^{\frac{3}{2}}u_{i} \partial_{i}\theta-u_{i}\Lambda^{\frac{3}{2}} \partial_{i}\theta\|_{L^{\frac{3}{2}}}
\|\Lambda^{\frac{3}{2}}\theta\|_{L^{3}}\nonumber\\
&\leq C\|\Lambda u\|_{L^{3}}\|\Lambda^{\frac{1}{2}} \nabla\theta\|_{L^{3}}
\|\Lambda^{\frac{3}{2}}\theta\|_{L^{3}}
\nonumber\\
&\leq C\|\omega\|_{L^{3}}\|\Lambda^{\frac{3}{2}}\theta\|_{L^{3}}^{2}
\nonumber\\
&\leq C\|\omega\|_{L^{3}}\|\Lambda^{\frac{3}{2}}\theta\|_{L^{2}}^{\frac{2}{3}}
\|\Delta\theta\|_{L^{2}}^{\frac{4}{3}}
\nonumber\\
&\leq
\frac{1}{8}\|\Delta\theta\|_{L^{2}}^{2}+C\|\omega\|_{L^{3}}^{3}
\|\Lambda^{\frac{3}{2}}\theta\|_{L^{2}}^{2}.\nonumber
\end{align}
Regarding $\Gamma_{2}$, one gets by \eqref{cmoi78ye2r} that
\begin{align}
|\Gamma_{2}|
&\leq C\|\Lambda u\|_{L^{3}}\|\Lambda^{\frac{1}{2}} (\nabla\theta \Lambda^{\frac{3}{2}}\theta)\|_{L^{\frac{3}{2}}}
\nonumber\\
&\leq C\|\Lambda u\|_{L^{3}}(\|\Lambda^{\frac{1}{2}} \nabla\theta\|_{L^{3}}
\|\Lambda^{\frac{3}{2}}\theta\|_{L^{3}}+\| \nabla\theta\|_{L^{6}}
\|\Delta\theta\|_{L^{2}})
\nonumber\\
&\leq
C\|\omega\|_{L^{3}}(\|\Lambda^{\frac{3}{2}}\theta\|_{L^{3}}^{2}+\| \Lambda^{\frac{2}{3}} \nabla\theta\|_{L^{2}}
\|\Delta\theta\|_{L^{2}})
\nonumber\\
&\leq C\|\omega\|_{L^{3}}(\|\Lambda^{\frac{3}{2}}\theta\|_{L^{2}}^{\frac{2}{3}}
\|\Delta\theta\|_{L^{2}}^{\frac{4}{3}}+\|\Lambda^{\frac{3}{2}}
\theta\|_{L^{2}}^{\frac{2}{3}}
\|\Delta\theta\|_{L^{2}}^{\frac{1}{3}}\|\Delta\theta\|_{L^{2}})
\nonumber\\
&\leq
\frac{1}{8}\|\Delta\theta\|_{L^{2}}^{2}+C\|\omega\|_{L^{3}}^{3}
\|\Lambda^{\frac{3}{2}}\theta\|_{L^{2}}^{2}.\nonumber
\end{align}
As a result, we have
\begin{align}\label{sdlpyt666}
\left|\int_{\mathbb{R}^{2}}
\Lambda^{\frac{3}{2}}\big(u \cdot
\nabla\theta\big)\Lambda^{\frac{3}{2}}\theta\,dx\right|\leq \frac{1}{4}\|\Delta\theta\|_{L^{2}}^{2}+C\|\omega\|_{L^{3}}^{3}
\|\Lambda^{\frac{3}{2}}\theta\|_{L^{2}}^{2}.
\end{align}
It should be pointed out that we are still able to deduce the above bound \eqref{sdlpyt666} without the use of \eqref{tdfhklsd2}.
Actually, invoking the easy computations and $\nabla\cdot u=0$, we get
\begin{align}
\int_{\mathbb{R}^{2}}
\Lambda^{\frac{3}{2}}\big(u \cdot
\nabla\theta\big)\Lambda^{\frac{3}{2}}\theta\,dx&=\int_{\mathbb{R}^{2}}
\Lambda^{\frac{3}{2}}(u_{i} \partial_{i}\theta)\Lambda^{\frac{3}{2}}\theta\,dx
\nonumber\\&=\int_{\mathbb{R}^{2}}
\Lambda (u_{i} \partial_{i}\theta)\Lambda^{2}\theta\,dx
\nonumber\\&=\int_{\mathbb{R}^{2}}
[\Lambda, u_{i}]\partial_{i}\theta \Lambda^{2}\theta\,dx+
\int_{\mathbb{R}^{2}}u_{i}\Lambda  \partial_{i}\theta\Lambda^{2}\theta\,dx
\nonumber\\&=\int_{\mathbb{R}^{2}}
[\Lambda, u_{i}]\partial_{i}\theta \Lambda^{2}\theta\,dx+
\int_{\mathbb{R}^{2}}\Lambda^{\frac{1}{2}}(u_{i}\Lambda  \partial_{i}\theta)\Lambda^{\frac{3}{2}}\theta\,dx
\nonumber\\&=\int_{\mathbb{R}^{2}}
[\Lambda, u_{i}]\partial_{i}\theta \Lambda^{2}\theta\,dx+
\int_{\mathbb{R}^{2}}[\Lambda^{\frac{1}{2}}, u_{i}]\Lambda  \partial_{i}\theta\Lambda^{\frac{3}{2}}\theta\,dx
\nonumber\\&\quad+
\int_{\mathbb{R}^{2}}u_{i}\Lambda^{\frac{1}{2}}\Lambda  \partial_{i}\theta\Lambda^{\frac{3}{2}}\theta\,dx
\nonumber\\&=\int_{\mathbb{R}^{2}}
[\Lambda, u_{i}]\partial_{i}\theta \Lambda^{2}\theta\,dx+
\int_{\mathbb{R}^{2}}[\Lambda^{\frac{1}{2}}, u_{i}]\Lambda  \partial_{i}\theta\Lambda^{\frac{3}{2}}\theta\,dx
\nonumber\\&\quad+
\int_{\mathbb{R}^{2}}u_{i}\partial_{i}\Lambda^{\frac{3}{2}}\theta\Lambda^{\frac{3}{2}}\theta\,dx
\nonumber\\&=\int_{\mathbb{R}^{2}}
[\Lambda, u_{i}]\partial_{i}\theta \Lambda^{2}\theta\,dx+
\int_{\mathbb{R}^{2}}[\Lambda^{\frac{1}{2}}, u_{i}]\Lambda  \partial_{i}\theta\Lambda^{\frac{3}{2}}\theta\,dx
\nonumber\\&\triangleq \widetilde{\Gamma}_{1}+\widetilde{\Gamma}_{2}.\nonumber
\end{align}
Taking advantage of \eqref{cmoi78ye1}, it yields
\begin{align}
|\widetilde{\Gamma}_{1}|&\leq C\|[\Lambda, u_{i}]\partial_{i}\theta\|_{L^{2}}
\|\Lambda^{2}\theta\|_{L^{2}}\nonumber\\
&\leq C(\|\nabla u\|_{L^{3}}\|\Lambda\theta\|_{L^{6}}+\|\nabla\theta\|_{L^{6}}\|\Lambda u\|_{L^{3}})
\|\Lambda^{2}\theta\|_{L^{2}}
\nonumber\\
&\leq C\|\omega\|_{L^{3}}\|\Lambda \theta\|_{L^{6}} \|\Delta\theta\|_{L^{2}}
\nonumber\\
&\leq C\|\omega\|_{L^{3}}\|\Lambda^{\frac{3}{2}}\theta\|_{L^{2}}^{\frac{2}{3}}
\|\Delta\theta\|_{L^{2}}^{\frac{1}{3}} \|\Delta\theta\|_{L^{2}}
\nonumber\\
&\leq
\frac{1}{8}\|\Delta\theta\|_{L^{2}}^{2}+C\|\omega\|_{L^{3}}^{3}
\|\Lambda^{\frac{3}{2}}\theta\|_{L^{2}}^{2}.\nonumber
\end{align}
Again, it follows from \eqref{cmoi78ye1} that
\begin{align}
|\widetilde{\Gamma}_{2}|
&\leq C\|[\Lambda^{\frac{1}{2}}, u_{i}]\Lambda  \partial_{i}\theta\|_{L^{\frac{12}{7}}}\|\Lambda^{\frac{3}{2}} \theta\|_{L^{\frac{12}{5}}}
\nonumber\\
&\leq C(\|\nabla u\|_{L^{3}}\|\Lambda^{\frac{3}{2}} \theta\|_{L^{4}}+\|\Lambda\nabla\theta\|_{L^{2}}\|\Lambda^{\frac{1}{2}}u\|_{L^{12}})
\|\Lambda^{\frac{3}{2}} \theta\|_{L^{\frac{12}{5}}}
\nonumber\\
&\leq C(\|\omega\|_{L^{3}}\|\Delta \theta\|_{L^{2}}+\|\Delta\theta\|_{L^{2}}\|\Lambda u\|_{L^{3}})
\|\Lambda^{\frac{3}{2}}\theta\|_{L^{2}}^{\frac{2}{3}}
\|\Delta\theta\|_{L^{2}}^{\frac{1}{3}}
\nonumber\\
&\leq C\|\omega\|_{L^{3}}\|\Delta \theta\|_{L^{2}}
\|\Lambda^{\frac{3}{2}}\theta\|_{L^{2}}^{\frac{2}{3}}
\|\Delta\theta\|_{L^{2}}^{\frac{1}{3}}
\nonumber\\
&\leq
\frac{1}{8}\|\Delta\theta\|_{L^{2}}^{2}+C\|\omega\|_{L^{3}}^{3}
\|\Lambda^{\frac{3}{2}}\theta\|_{L^{2}}^{2}.\nonumber
\end{align}
Consequently, just using \eqref{cmoi78ye1}, we still obtain the desired bound \eqref{sdlpyt666}.
Putting the above estimates \eqref{yghp854} and \eqref{sdlpyt666} into \eqref{yrtyp01}, it yields
\begin{align}\label{yrtyp02}
 \frac{d}{dt}\|\Lambda^{\frac{3}{2}}\theta(t)\|_{L^{2}}^{2}
+\|\Delta\theta\|_{L^{2}}^{2}\leq C\|\omega\|_{L^{3}}^{3}
\|\Lambda^{\frac{3}{2}}\theta\|_{L^{2}}^{2}+C\|\omega\|_{L^{2}}^{2}.
\end{align}
Making use of \eqref{tdkp329} and applying the Gronwall inequality to \eqref{yrtyp02}, we thus conclude the desired bound \eqref{yrtyp03}. With \eqref{tdkp329} and \eqref{yrtyp03} in hand, we are ready to prove Theorem \ref{addTh4}. To this end, applying $(\Lambda^{s},\Lambda^{\widetilde{s}})$ to $((\ref{BEBou})_{1},(\ref{BEBou})_{2})$ and taking $L^{2}$ inner product with $(\Lambda^{s}u,\Lambda^{\widetilde{s}}\theta)$, we obtain
\begin{align}\label{bhgfty68d}
&\frac{1}{2}\frac{d}{dt}(\|\Lambda^{s}u(t)\|_{L^{2}}^{2}
+\|\Lambda^{\widetilde{s}}\theta(t)\|_{L^{2}}^{2}) +\|\Lambda^{\widetilde{s}+\frac{1}{2}}\theta\|_{L^{2}}^{2}\nonumber\\&= \int_{\mathbb{R}^{2}}{\Lambda^{s}(\theta e_{2})\cdot
\Lambda^{s}u\,dx}+\int_{\mathbb{R}^{2}}{\Lambda^{\widetilde{s}}u_{2}
\Lambda^{\widetilde{s}}\theta\,dx}\nonumber\\&\quad-
\int_{\mathbb{R}^{2}}{\Lambda^{\widetilde{s}}(u\cdot\nabla\theta)
\Lambda^{\widetilde{s}}\theta\,dx}  -\int_{\mathbb{R}^{2}}{\Lambda^{s}(
u\cdot\nabla u)\cdot \Lambda^{s}u\,dx}.
\end{align}
We recall the basic $L^2$-energy estimate \eqref{ghvber68}, namely,
\begin{align}\label{bhgfty68f}
\frac{1}{2}\frac{d}{dt}(\|u(t)\|_{L^{2}}^{2}
+\|\theta(t)\|_{L^{2}}^{2}) + \|\Lambda^{ \frac{1}{2}}\theta\|_{L^2}^{2}\leq \|u\|_{L^{2}}^{2}+\|\theta\|_{L^{2}}^{2}.
\end{align}
It follows by combining \eqref{bhgfty68d} and \eqref{bhgfty68f} that
\begin{align}\label{tdkp330}
&\frac{1}{2}\frac{d}{dt}(\|u(t)\|_{H^{s}}^{2}
+\|\theta(t)\|_{H^{\widetilde{s}}}^{2}) + \|\Lambda^{\frac{1}{2}}\theta\|_{H^{\widetilde{s}}}^{2}\nonumber\\&\leq \int_{\mathbb{R}^{2}}{\Lambda^{s}(\theta e_{2})\cdot
\Lambda^{s}u\,dx}+\int_{\mathbb{R}^{2}}{\Lambda^{\widetilde{s}}u_{2}
\Lambda^{\widetilde{s}}\theta\,dx}-
\int_{\mathbb{R}^{2}}{\Lambda^{\widetilde{s}}(u\cdot\nabla\theta)
\Lambda^{\widetilde{s}}\theta\,dx} \nonumber\\&\quad -\int_{\mathbb{R}^{2}}{\Lambda^{s}(
u\cdot\nabla u)\cdot \Lambda^{s}u\,dx}+\|u\|_{L^{2}}^{2}+\|\theta\|_{L^{2}}^{2}.
\end{align}
In order to bound the first two terms at the right hand side of \eqref{tdkp330}, we need the restriction $|s-\widetilde{s}|\leq\frac{1}{2}$.
In fact, according to the direct embedding, one has
\begin{align}
\left|\int_{\mathbb{R}^{2}}{\Lambda^{s}(\theta e_{2})\cdot
\Lambda^{s}u\,dx}\right|\leq & C\|\Lambda^{\widetilde{s}+\frac{1}{2}}\theta\|_{L^{2}}
\|\Lambda^{2s-\widetilde{s}-\frac{1}{2}}u\|_{L^2}
\nonumber\\
\leq & C\|\Lambda^{\frac{1}{2}}\theta\|_{H^{\widetilde{s}}}\|u\|_{H^{s}}
\nonumber\\
 \leq&\frac{1}{4}\|\Lambda^{\frac{1}{2}}\theta\|_{H^{\widetilde{s}}}^{2}
 +C\|u\|_{H^{s}}^{2},\nonumber
\end{align}
where $s$ and $\widetilde{s}$ should satisfy
\begin{align}\label{tupp011}
s-\frac{1}{2}\leq\widetilde{s}\leq 2s-\frac{1}{2}.
\end{align}
Similarly, we obtain
\begin{align}
\left|\int_{\mathbb{R}^{2}}{\Lambda^{\widetilde{s}}u_{2}
\Lambda^{\widetilde{s}}\theta\,dx}\right|\leq & C\|\Lambda^{\widetilde{s}+\frac{1}{2}}\theta\|_{L^{2}}
\|\Lambda^{\widetilde{s}-\frac{1}{2}}u\|_{L^2}
\nonumber\\
\leq & C\|\Lambda^{\frac{1}{2}}\theta\|_{H^{\widetilde{s}}}\|u\|_{H^{s}}
\nonumber\\
 \leq&\frac{1}{4}\|\Lambda^{\frac{1}{2}}\theta\|_{H^{\widetilde{s}}}^{2}
 +C\|u\|_{H^{s}}^{2},\nonumber
\end{align}
where $s$ and $\widetilde{s}$ should satisfy
\begin{align}\label{tupp022}
\frac{1}{2}\leq\widetilde{s}\leq s+\frac{1}{2}.
\end{align}
Thanks to $s>2$, combining \eqref{tupp011} and \eqref{tupp022} implies that $s$ and $\widetilde{s}$ should satisfy $|s-\widetilde{s}|\leq\frac{1}{2}$.
Due to $|s-\widetilde{s}|\leq\frac{1}{2}$, the third term at the right hand side of \eqref{tdkp330} is difficult to handle. To this end, we should rewrite it by $\nabla\cdot u=0$ that
\begin{align}
\int_{\mathbb{R}^{2}}{\Lambda^{\widetilde{s}}(u\cdot\nabla\theta)
\Lambda^{\widetilde{s}}\theta\,dx}
 =&\int_{\mathbb{R}^{2}}{\Lambda^{\widetilde{s}}(u_{i}\partial_{i}\theta)
\Lambda^{\widetilde{s}}\theta\,dx}
 \nonumber\\ =&\int_{\mathbb{R}^{2}}{\Lambda^{\widetilde{s}-\frac{1}{2}}(u_{i}\partial_{i}\theta)
\Lambda^{\widetilde{s}+\frac{1}{2}}\theta\,dx} \nonumber\\ =&\int_{\mathbb{R}^{2}}{[\Lambda^{\widetilde{s}-\frac{1}{2}},\,u_{i}]
\partial_{i}\theta
\Lambda^{\widetilde{s}+\frac{1}{2}}\theta\,dx}
+\int_{\mathbb{R}^{2}}{u_{i}\Lambda^{\widetilde{s}-\frac{1}{2}}
\partial_{i}\theta
\Lambda^{\widetilde{s}+\frac{1}{2}}\theta\,dx}
 \nonumber\\ =&\int_{\mathbb{R}^{2}}{[\Lambda^{\widetilde{s}-\frac{1}{2}},\,u_{i}]
\partial_{i}\theta
\Lambda^{\widetilde{s}+\frac{1}{2}}\theta\,dx}
+\int_{\mathbb{R}^{2}}{\Lambda^{\frac{1}{2}}(u_{i}\Lambda^{\widetilde{s}-\frac{1}{2}}
\partial_{i}\theta)
\Lambda^{\widetilde{s}}\theta\,dx}
 \nonumber\\ =&\int_{\mathbb{R}^{2}}{[\Lambda^{\widetilde{s}-\frac{1}{2}},\,u_{i}]
\partial_{i}\theta
\Lambda^{\widetilde{s}+\frac{1}{2}}\theta\,dx}
+\int_{\mathbb{R}^{2}}{[\Lambda^{\frac{1}{2}},\,u_{i}]\Lambda^{\widetilde{s}-\frac{1}{2}}
\partial_{i}\theta
\Lambda^{\widetilde{s}}\theta\,dx}\nonumber\\&
\int_{\mathbb{R}^{2}}{u_{i}\Lambda^{\frac{1}{2}}\Lambda^{\widetilde{s}-\frac{1}{2}}
\partial_{i}\theta
\Lambda^{\widetilde{s}}\theta\,dx}
\nonumber\\ =&\int_{\mathbb{R}^{2}}{[\Lambda^{\widetilde{s}-\frac{1}{2}},\,u_{i}]
\partial_{i}\theta
\Lambda^{\widetilde{s}+\frac{1}{2}}\theta\,dx}
+\int_{\mathbb{R}^{2}}{[\Lambda^{\frac{1}{2}},\,u_{i}]\Lambda^{\widetilde{s}-\frac{1}{2}}
\partial_{i}\theta
\Lambda^{\widetilde{s}}\theta\,dx}\nonumber\\&
+\int_{\mathbb{R}^{2}}{u_{i} \partial_{i}\Lambda^{\widetilde{s}}
\theta
\Lambda^{\widetilde{s}}\theta\,dx}\nonumber\\ =&\int_{\mathbb{R}^{2}}{[\Lambda^{\widetilde{s}-\frac{1}{2}},\,u_{i}]
\partial_{i}\theta
\Lambda^{\widetilde{s}+\frac{1}{2}}\theta\,dx}
+\int_{\mathbb{R}^{2}}{[\Lambda^{\frac{1}{2}},\,u_{i}]\Lambda^{\widetilde{s}-\frac{1}{2}}
\partial_{i}\theta
\Lambda^{\widetilde{s}}\theta\,dx}.\nonumber
\end{align}
This along with (\ref{cmoi78ye1}) implies
\begin{align}
\left|
\int_{\mathbb{R}^{2}}{\Lambda^{\widetilde{s}}(u\cdot\nabla\theta)
\Lambda^{\widetilde{s}}\theta\,dx}\right|
 \leq&
 \left|\int_{\mathbb{R}^{2}}{[\Lambda^{\widetilde{s}-\frac{1}{2}},\,u_{i}]
\partial_{i}\theta
\Lambda^{\widetilde{s}+\frac{1}{2}}\theta\,dx} \right|+ \left|\int_{\mathbb{R}^{2}}{[\Lambda^{\frac{1}{2}},\,u_{i}]\Lambda^{\widetilde{s}-\frac{1}{2}}
\partial_{i}\theta
\Lambda^{\widetilde{s}}\theta\,dx}\right|
 \nonumber\\ \leq&
\|[\Lambda^{\widetilde{s}-\frac{1}{2}},\,u_{i}]
\partial_{i}\theta\|_{L^{2}}
\|\Lambda^{\widetilde{s}+\frac{1}{2}}\theta\|_{L^{2}}
+\|[\Lambda^{\frac{1}{2}},\,u_{i}]\Lambda^{\widetilde{s}-\frac{1}{2}}
\partial_{i}\theta\|_{L^{\frac{3}{2}}}
\|\Lambda^{\widetilde{s}}\theta\|_{L^{3}}
\nonumber\\ \leq& C(\|\nabla u\|_{L^{4}}\| \Lambda^{\widetilde{s}-\frac{1}{2}}\theta\|_{L^{4}}+\|\nabla \theta\|_{L^{\infty}}\| \Lambda^{\widetilde{s}-\frac{1}{2}}u\|_{L^{2}})
\|\Lambda^{\widetilde{s}+\frac{1}{2}}\theta\|_{L^{2}}
 \nonumber\\& +
C(\|\nabla u\|_{L^{6}}\| \Lambda^{\widetilde{s}}\theta\|_{L^{2}}+\|\Lambda^{\widetilde{s}-\frac{1}{2}}
\partial_{i}\theta\|_{L^{2}}\| \Lambda^{\frac{1}{2}}u\|_{L^{6}})
\|\Lambda^{\widetilde{s}}\theta\|_{L^{3}}\nonumber\\ \leq& C(\|\omega\|_{L^{4}}\| \Lambda^{\widetilde{s}}\theta\|_{L^{2}}+\|\nabla \theta\|_{L^{\infty}}\| \Lambda^{\widetilde{s}-\frac{1}{2}}u\|_{L^{2}})
\|\Lambda^{\widetilde{s}+\frac{1}{2}}\theta\|_{L^{2}}
 \nonumber\\& +
C(\|\omega\|_{L^{6}}\| \Lambda^{\widetilde{s}}\theta\|_{L^{2}}+\|\Lambda^{\widetilde{s}+\frac{1}{2}}
\theta\|_{L^{2}}\|\omega\|_{L^{\frac{12}{5}}})
\|\Lambda^{\widetilde{s}}\theta\|_{L^{2}}^{\frac{1}{3}}\|\Lambda^{\widetilde{s}+\frac{1}{2}}
\theta\|_{L^{2}}^{\frac{2}{3}}\nonumber\\
 \leq&\frac{1}{4}\|\Lambda^{\widetilde{s}+\frac{1}{2}}\theta\|_{L^{2}}^{2}
 +C(\|\omega\|_{L^{4}}^{2}+\|\omega\|_{L^{6}}^{\frac{3}{2}}+\|\omega\|_{L^{\frac{12}{5}}}^{6})\| \Lambda^{\widetilde{s}}\theta\|_{L^{2}}^{2} \nonumber\\& +C\|\nabla \theta\|_{L^{\infty}}^{2}\| \Lambda^{\widetilde{s}-\frac{1}{2}}u\|_{L^{2}}^{2}
\nonumber\\
 \leq&\frac{1}{4}
 \|\Lambda^{\frac{1}{2}}\theta\|_{H^{\widetilde{s}}}^{2}+C(\|\omega\|_{L^{4}}^{2}+\|\omega\|_{L^{6}}^{\frac{3}{2}}+\|\omega\|_{L^{\frac{12}{5}}}^{6})\| \|\theta\|_{H^{\widetilde{s}}}^{2} \nonumber\\& +C\|\nabla \theta\|_{L^{\infty}}^{2}\| \|u\|_{H^{s}}^{2},\nonumber
\end{align}
where in the last line we have used the following fact due to $\widetilde{s}\leq s+\frac{1}{2}$
$$\|\Lambda^{\widetilde{s}-\frac{1}{2}}u\|_{L^{2}}\leq \|u\|_{H^{\widetilde{s}-\frac{1}{2}}}\leq  \|u\|_{H^{s}}.$$
Thanks to $\nabla\cdot u=0$ and (\ref{cmoi78ye1}), we have
\begin{align}
\left|\int_{\mathbb{R}^{2}}{\Lambda^{s}(
u\cdot\nabla u)\cdot \Lambda^{s}u\,dx}\right|
 =&
\left|
\int_{\mathbb{R}^{2}}{[\Lambda^{s},\,u\cdot\nabla]u\cdot
\Lambda^{s}u\,dx} \right|
 \nonumber\\ \leq&
\|[\Lambda^{s}, u\cdot\nabla]u\|_{L^{2}}\|\Lambda^{s}u\|_{L^{2}}
\nonumber\\
 \leq&C\|\nabla u\|_{L^{\infty}}\|\Lambda^{s}u\|_{L^{2}}^{2}.
 \nonumber
 \end{align}
Putting all the above estimates into \eqref{tdkp330} leads to
 \begin{align}\label{hnkpt11}
& \frac{d}{dt}(\|u(t)\|_{H^{s}}^{2}
+\|\theta(t)\|_{H^{\widetilde{s}}}^{2}) + \|\Lambda^{\frac{1}{2}}\theta\|_{H^{\widetilde{s}}}^{2}\nonumber\\ &\leq
C(1+\|\nabla u\|_{L^{\infty}}+\|\omega\|_{L^{4}}^{2}+\|\omega\|_{L^{6}}^{\frac{3}{2}}+\|\omega\|_{L^{\frac{12}{5}}}^{6})
(\|u\|_{H^{s}}^{2}
+\|\theta\|_{H^{\widetilde{s}}}^{2})\nonumber\\&\quad+C\|\nabla \theta\|_{L^{\infty}}^{2}
(\|u\|_{H^{s}}^{2}
+\|\theta\|_{H^{\widetilde{s}}}^{2}).
\end{align}
Thanks to the following logarithmic Sobolev inequalities (see
\cite{Kozono2} for instance)
\begin{eqnarray}
\|\nabla u\|_{L^{\infty}} \leq C\Big(1+\|u\|_{L^{2}}+ \|\omega\|_{L^{\infty}} \ln\big(e+\|\Lambda^{\sigma}
u\|_{L^{2}}\big)\Big),\quad \sigma>2,\nonumber
\end{eqnarray}
\begin{eqnarray}
\|\nabla \theta\|_{L^{\infty}} \leq C\Big(1+\|\theta\|_{L^{2}}+ \|\Delta
\theta\|_{L^{2}} \sqrt{\ln\big(e+\|\Lambda^{\sigma}
\theta\|_{L^{2}}\big)}\Big),\quad \sigma>2,\nonumber
\end{eqnarray}
we deduce from \eqref{hnkpt11} that
 \begin{align}\label{hnkpt12}
& \frac{d}{dt}(\|u(t)\|_{H^{s}}^{2}
+\|\theta(t)\|_{H^{\widetilde{s}}}^{2}) + \|\Lambda^{\frac{1}{2}}\theta\|_{H^{\widetilde{s}}}^{2}\nonumber\\ &\leq
C(1+\|\omega\|_{L^{\infty}}+\|\omega\|_{L^{4}}^{2}+\|\omega\|_{L^{6}}^{\frac{3}{2}}
+\|\omega\|_{L^{\frac{12}{5}}}^{6})
\ln\big(e+\|u\|_{H^{s}}\big)(\|u\|_{H^{s}}^{2}
+\|\theta\|_{H^{\widetilde{s}}}^{2})\nonumber\\&\quad+C\|\Delta\theta\|_{L^{2}}^{2}
\ln\big(e+\|\Lambda^{\frac{1}{2}}\theta\|_{H^{\widetilde{s}}}\big)
(\|u\|_{H^{s}}^{2}
+\|\theta\|_{H^{\widetilde{s}}}^{2}).
\end{align}
Obviously, one may check that
\begin{align} \label{hnkpt13}
\|\Delta\theta\|_{L^{2}}^{2}&\leq
C\|\theta\|_{L^{2}}^{\frac{2(2\widetilde{s}-3)}{2\widetilde{s}+1}}\|\Lambda^{\widetilde{s}+\frac{1}{2}}
\theta\|_{L^{2}}^{\frac{8}{2\widetilde{s}+1}}\nonumber\\&\leq
C(\|\Lambda^{\frac{1}{2}}\theta\|_{H^{\widetilde{s}}}^{2})^{\frac{4}{2\widetilde{s}+1}}.
\end{align}
Due to $\widetilde{s}\geq s-\frac{1}{2}>\frac{3}{2}$, we have that $\frac{4}{2\widetilde{s}+1}<1$. Moreover, $\|\Delta\theta(t)\|_{L^{2}}^{2}$ belongs to $L_{loc}^{1}(\mathbb{R}_{+})$ due to \eqref{yrtyp03}.
Keeping in mind \eqref{hnkpt12}-\eqref{hnkpt13} and recalling \eqref{tdkp329}-\eqref{yrtyp03}, we obtain by
applying Lemma \ref{Lem01khj} to \eqref{hnkpt12} that
\begin{align}\label{dhnkgp16}
\|u(t)\|_{H^{s}}^{2}
+\|\theta(t)\|_{H^{\widetilde{s}}}^{2}+\int_{0}^{t}
\|\Lambda^{\frac{1}{2}}\theta(\tau)\|_{H^{\widetilde{s}}}^{2} \,d\tau\leq
C(t,\,u_{0},\,\theta_{0}).
\end{align}
With higher regularity estimate \eqref{dhnkgp16} in hand, we thus have
\begin{align}
\int_{0}^{t}
(\|\nabla u(\tau)\|_{L^{\infty}}+\|\nabla \theta(\tau)\|_{L^{\infty}} )\,d\tau\leq \int_{0}^{t}
(\|u(\tau)\|_{H^{s}}+\|\Lambda^{\frac{1}{2}}\theta(\tau)\|_{H^{\widetilde{s}}}) \,d\tau\leq
C(t,\,u_{0},\,\theta_{0}),\nonumber
\end{align}
which implies the uniqueness immediately.
Consequently, we complete the proof of Theorem \ref{addTh4}.
\end{proof}

\vskip .3in
\section*{Acknowledgements}
This work was supported by the Qing Lan Project of Jiangsu Province.


\vskip .3in
\begin{thebibliography}{00} \frenchspacing


\bibitem{ACW11}
D. Adhikari, C. Cao, J. Wu, \emph{Global regularity results for the 2D Boussinesq equations with vertical dissipation}, J. Differential Equations, \textbf{251} (2011), 1637--1655.

\bibitem{ATGKZMH}
M. Altaf, E. Titi, T. Gebrael, O. Knio, L. Zhao, M. McCabe, I. Hoteit, \emph{Downscaling the 2D B$\rm \acute{e}$nard convection equations using continuous data assimilation}, Comput. Geosci. \textbf{21} (3) (2017), 393--410.


\bibitem{AP}
A. Ambrosetti, G. Prodi, \emph{A Primer of Nonlinear Analysis}. Cambridge Studies in Advanced Mathematics 34, Cambridge, Cambridge Univ. Press, 1995.


\bibitem{BCD}
H. Bahouri, J.-Y. Chemin, R. Danchin, \emph{Fourier Analysis and Nonlinear
Partial Differential Equations}, Grundlehren der mathematischen
Wissenschaften, 343, Springer (2011).

\bibitem{CLTPD}
C. Cao, J. Li, E. Titi, \emph{Global well-posedness of the 3D primitive equations with horizontal viscosity and vertical diffusivity}, Phys. D 412 (2020), 132606, 25 pp.

\bibitem{CJEWNO}
Y. Cao, M. Jolly, E. Titi, J. Whitehead, \emph{Algebraic bounds on the Rayleigh-B$\rm \acute{e}$nard attractor}, Nonlinearity \textbf{34} (1) (2021), 509--531.

\bibitem{Chae2}
D. Chae, \emph{Global regularity for the 2D Boussinesq equations with
partial viscosity terms}, Adv. Math.  \textbf{{203}} (2006), 497-513.


\bibitem{Chan81}
S. Chandrasekhar, \emph{Hydrodynamic and Hydromagnetic Stability}, Dover Publications, Inc., 1981.

\bibitem{CMZ}
Q. Chen, C. Miao, Z. Zhang, \emph{A new Bernstein inequality and the 2D
dissipative quasigeostrophic equation}, Comm. Math. Phys. \textbf{271} (2007),
821--838.

\bibitem{Constantin}
P. Constantin, \emph{Energy spectrum of quasigeostrophic turbulence}, Phys. Rev.
Lett. 89(18), 184501 (2002).

\bibitem{ConDoering}
P. Constantin, C. Doering, \emph{Infinite Prandtl number convection}, J. Statist. Phys. \textbf{94} (1999),
159--172.

\bibitem{DP2}
R. Danchin, M. Paicu, \emph{Global well-posedness issues for the inviscid Boussinesq system with Yudovich's type data}, Comm. Math. Phys. \textbf{290} (2009), 1--14.

\bibitem{DPqq}
R. Danchin, M. Paicu, \emph{Global existence results for the anisotropic Boussinesq system in dimension two}, Math. Models Methods Appl. Sci. \textbf{21} (2011), 421--457.

\bibitem{Eringen}
A. Eringen, \emph{Theory of micropolar fluids}, J. Math. Mech. \textbf{16} (1966), 1--16.


\bibitem{FJT}
A. Farhat, M. Jolly, E. Titi, \emph{Continuous data assimilation for the 2D B$\rm \acute{e}$nard convection through velocity measurements alone}, Phys. D \textbf{303} (2015), 59--66.

\bibitem{FLTJNS}
A. Farhat, E. Lunasin, E. Titi, \emph{Continuous data assimilation for a 2D B$\rm \acute{e}$nard convection system through horizontal velocity measurements alone}, J. Nonlinear Sci. \textbf{27}  (2017), 1065--1087.



\bibitem{FMT}
C. Foias, O. Manley, R. Temam, \emph{Attractors for the B$\rm \acute{e}$nard problem: existence
and physical bounds on their fractal dimension}, Nonlinear Anal. \textbf{11}
(1987), 939--967.



\bibitem{Fujiwara}
K. Fujiwara, V. Georgiev, T. Ozawa, \emph{Higher order fractional Leibniz rule}, J. Fourier Anal. Appl. 24 (2018), 650--665.

\bibitem{Gpadd}
G. Galdi, M. Padula, \emph{A new approach to energy theory in the stability of fluid motion}, Arch. Ration. Mech. Anal. \textbf{110}
(1990), 187--286.



\bibitem{Gill}
A. Gill, \emph{Atmosphere-Ocean Dynamics}, International Geophysics Series, Academic Press, 30,
1982.

\bibitem{Hmdapde}
T. Hmidi, \emph{On a maximum principle and its application to the logarithmically critical Boussinesq system}, Anal. PDE \textbf{4} (2011), 247--284.


\bibitem{HKbmp1}
T. Hmidi, S. Keraani, \emph{On the global well-posedness of the Boussinesq system with zero viscosity}, Indiana Univ. Math. J. \textbf{58} (2009), 1591--1618


\bibitem{HK3}
T. Hmidi, S. Keraani, F. Rousset, \emph{Global well-posedness for a
Boussinesq-Navier-Stokes system with critical dissipation}, J.
Differential Equations \textbf{249} (2010), 2147--2174.

\bibitem{HK4}
T. Hmidi, S. Keraani, F. Rousset, \emph{Global well-posedness for
Euler-Boussinesq system with critical dissipation}, Comm. Partial
Differential Equations \textbf{36} (2011), 420--445.

\bibitem{HZ10}
T. Hmidi, M. Zerguine, \emph{On the global well-posedness of the Euler-Boussinesq system with fractional dissipation}, Phys. D  \textbf{{239}}  (2010), 1387--1401.

\bibitem{JN}
N. Ju, \emph{The maximum principle and the global attractor for the dissipative 2D quasi-geostrophic equations}, Comm. Math.
Phys. \textbf{255} (2005), 161--181.

\bibitem{TKG}
T. Kato, G. Ponce, \emph{Commutator estimates and the Euler and Navier-Stokes equations},  Comm. Pure. Appl. Math.  \textbf{41} (1988), 891--907.

\bibitem{KCRTW}
D. KC, D. Regmi, L. Tao, J. Wu, \emph{Generalized 2D Euler-Boussinesq equations with a singular velocity, J. Differential Equations}, \textbf{257} (2014), 82--108.

\bibitem{Kenig}
C. E. Kenig, G. Ponce, L. Vega, \emph{Well-posedness and scattering results for the generalized
Korteweg-de-Vries equation via the contraction principle}, Comm. Pure App. Math. \textbf{46} (1993), 527--620.


\bibitem{Kozono2}
H. Kozono, T. Ogawa, Y. Taniuchi, \emph{The critical Sobolev
inequalities in Besov spaces and regularity criterion to some
semi-linear evolution equations}, Math. Z. \textbf{242} (2002), 251--278.

\bibitem{Lirmi19}
D. Li,  \emph{On Kato-Ponce and fractional Leibniz},
Rev. Mat. Iberoam. \textbf{35} (2019), 23--100.


\bibitem{Lukasz1}
G. Lukaszewicz, \emph{Micropolar Fluids-Theory and Applications, Birkh$\rm\ddot{a}$user Basel}, 1999.

\bibitem{Lukasz2}
G. Lukaszewicz, \emph{Long time behavior of 2D micropolar fluid flows}, Math. Comput. Modell. \textbf{34} (2001), 487--509.



\bibitem{MW}
T. Ma, S. Wang, \emph{Rayleigh B$\rm \acute{e}$nard convection: dynamics and structure in the physical space}, Commun. Math. Sci. \textbf{5} (2007), 553--574.

\bibitem{MB}
A. Majda, A. Bertozzi, \emph{Vorticity and Incompressible
Flow}, Cambridge University Press, Cambridge, 2001.

\bibitem{Mulone}
G. Mulone, S. Rionero, \emph{Necessary and sufficient conditions for nonlinear stability in the magnetic B$\rm \acute{e}$nard problem}, Arch.
Ration. Mech. Anal. \textbf{166} (2003), 197--218.

\bibitem{PG}
J. Pedlosky, \emph{Geophysical fluid dynamics}, New York, Springer-Verlag,
1987.

\bibitem{Ra}
P. Rabinowitz, \emph{Existence and nonuniqueness of rectangular solutions of the B$\rm \acute{e}$nard problem}, Arch. Ration. Mech. Anal. \textbf{29}
(1968), 32--57.

\bibitem{RTemam}
R. Temam, \emph{Navier--Stokes Equations}, revised version, Stud. Math. Appl., vol. 2, North-Holland, 1979.

\bibitem{RTemam97}
R. Temam, \emph{Infinite-Dimensional Dynamical Systems in Mechanics and Physics}, 2nd edition, Applied
Mathematical Sciences, vol. 68. Springer, New York (1997).


\bibitem{WuXue2012}
G. Wu, L. Xue, \emph{Global well-posedness for the 2D inviscid B$\rm\acute{e}$nard system with fractional diffusivity and Yudovich's type data}, J. Differential Equations \textbf{253} (2012), 100--125.

\bibitem{XXue1}
X. Xu, L. Xue, \emph{Yudovich type solution for the 2D inviscid Boussinesq
system with critical and supercritical dissipation}, J. Differential
Equations \textbf{{256}} (2014), 3179--3207.

\bibitem{Yena14}
Z. Ye, \emph{Blow-up criterion of smooth solutions for the Boussinesq equations}, Nonlinear Anal. \textbf{110} (2014), 97--103.


\bibitem{Yena17}
Z. Ye, \emph{Regularity criterion of the 2D B$\rm\acute{e}$nard equations with critical and supercritical dissipation}, Nonlinear Anal. \textbf{156} (2017), 111--143.


\bibitem{Yeacap18}
Z. Ye, \emph{Some new regularity criteria for the 2D Euler-Boussinesq equations via the temperature}, Acta Appl. Math. \textbf{157} (2018), 141--169.

\bibitem{Yena}
Z. Ye, \emph{An alternative approach to global regularity for the 2D Euler-Boussinesq equations with critical dissipation}, Nonlinear Anal. \textbf{190} (2020), 111591, 5 pp.


\bibitem{Yudovich}
V. Yudovich, \emph{Non-stationary flows of an ideal incompressible fluid}, Akademija Nauk SSSR.
$\rm\check{Z}$urnal Vy$\rm\check{c}$islitel'no$\rm\check{1}$ Matematiki i Matemati$\rm\check{c}$esko$\rm\check{1}$ Fiziki 3 (1963), 1032--1066.
\end{thebibliography}
\end{document}